\documentclass{article}

\usepackage[english]{babel}
\usepackage[utf8]{inputenc}
\usepackage{csquotes}
\usepackage{amsmath, amssymb, xfrac, bm, bbm, mathtools, amsthm, mathabx}
\usepackage{parskip}
\usepackage{graphicx}
\usepackage{booktabs}
\usepackage{multirow}
\usepackage{enumitem}
\usepackage{hyperref}
\usepackage{array}
\usepackage{algorithm}
\usepackage{algpseudocode}
\usepackage{amsthm}
\usepackage{tikz}
\usetikzlibrary{fit,positioning,arrows.meta, calc}
\usepackage{caption}

\usepackage{float}
\usepackage[labelfont=bf]{caption, subcaption}

\usepackage[authoryear,round]{natbib}
\usepackage[top=1.5cm, left=2cm, right=2cm, bottom=3cm]{geometry}
\usepackage{xcolor}

\DeclareMathOperator{\argmin}{\arg\!\min}

\providecommand{\keywords}[1]
{
  \small	
  \textbf{Key words:} #1
}

\newtheorem{proposition}{Proposition}[section]

\newtheorem{definition}{Definition}[section]
\theoremstyle{remark}

\usepackage{tikz}
    \usetikzlibrary{arrows}
\usepackage{pgf}

\title{Inverse Optimization Without Inverse Optimization:\\
Direct Solution Prediction with Transformer Models}

\author{Macarena Navarro$^{1}$
 \and Willem-Jan van Hoeve$^{1}$
 \and Karan Singh$^{1}$}

\date{$^{1}$Tepper School of Business, Carnegie Mellon University}

\allowdisplaybreaks

\begin{document}
\maketitle

\begin{abstract}
    We present an end-to-end framework for generating solutions to combinatorial optimization problems with unknown components using transformer-based sequence-to-sequence neural networks. Our framework learns directly from past solutions and incorporates the known components, such as hard constraints, via a constraint reasoning module, yielding a constrained learning scheme.  The trained model generates new solutions that are structurally similar to past solutions and are guaranteed to respect the known constraints. We apply our approach to three combinatorial optimization problems with unknown components: the knapsack problem with an unknown reward function, the bipartite matching problem with an unknown objective function, and the single-machine scheduling problem with release times and unknown precedence constraints, with the objective of minimizing average completion time. We demonstrate that transformer models have remarkably strong performance and often produce near-optimal solutions in a fraction of a second. They can be particularly effective in the presence of more complex underlying objective functions and unknown implicit constraints compared to an LSTM-based alternative and inverse optimization.
\end{abstract}

\keywords{Inverse Optimization, Generative Neural Networks, Transformer Models, Constrained Learning, Structured Prediction} 

\section{Introduction}\label{sec:Intro}

Traditional optimization methods aim to find optimal solutions to decision problems that are specified by an objective function and a set of constraints.  Inverse optimization (IO) considers the reverse process, in which we are given a set of historical decisions and aim to find model parameters which render the historical decisions feasible and optimal~\citep{chan2025inverse,bodur2022inverse}. While IO has been studied for decades, there has been a recent surge in interest due to the abundance of available historical data coupled with the rise of data-driven IO.

Consider a classical optimization model of the form
$\min\limits_x \{ f(x) : x \in \mathcal{X} \}$, where $\mathcal{X}$ represents the feasible set of the problem.  Some parameters of the model
may be known and determined by the context or input data, while others may be estimated from historical solutions.  
For known parameters $u$ representing the context and estimated parameters $\theta$, the classical problem can be reformulated as the following {\em forward problem}:
\begin{equation}\label{eq:FP}
F(u, \theta) = \min\limits_x \{ f(x,u,\theta) : x \in \mathcal{X}(u,\theta) \},
\end{equation}
where the objective function and the feasible set are parametrized by $u \in \mathcal{U}$ and $\theta \in \Theta$.  To estimate $\theta$ we assume past data is available, of the form $\{ (\hat{x}_i, \hat{u}_i) \}_{i=1}^N$, representing $N$ historical solutions and their contexts. 
We let $\mathcal{X}^{\mathrm{opt}}(u, \theta) := \argmin\limits_x \{ f(x,u,\theta) : x \in \mathcal{X}(u,\theta)\}$ be the optimal solution set given $u$ and $\theta$.
Using the past data, we can estimate values for $\theta$ by solving the {\em inverse problem}:
\begin{equation}\label{eq:inv-prob}
\min\limits_\theta \{ g(\theta) : \theta \in \Theta^{\mathrm{inv}}(\hat{u}_i, \hat{x}_i) \,
\forall i \in \{1, \dots, N\}, \theta \in \Theta \},
\end{equation}
where $g$ is an application-specific objective function, and $\Theta^{\mathrm{inv}}(\hat{u}_i, \hat{x}_i) := \{ \theta : \hat{x}_i \in \mathcal{X}^{\mathrm{opt}}(\hat{u}_i, \theta) \}$ is the inverse-feasible set.

The above definition of IO assumes a perfect fit for $\theta$, in that it strictly enforces inverse feasibility.  In many practical applications this assumption is not realistic, because the data and past solutions are often ambiguous, partially inconsistent, or depending on human judgment.  For example, in the context of vehicle routing an optimal route may not only depend on measures such as distance or cost, but potentially also on driver preferences~\citep{merchan20242021}.
For these reasons, and motivated by the increasing availability of decision data, the recent literature has shifted towards data-driven parameter estimation that minimizes a loss function $\ell$ that penalizes a measure of violation of inverse feasibility (where $\kappa \geq 0$ trades off the original objective function and the loss):
\begin{equation}\label{eq:data-driven-IO}
\min\limits_\theta \left\{  \kappa g(\theta) 
+ \frac{1}{N} \sum_{i=1}^N \ell\left(\hat{x}_i, \mathcal{X}^{\mathrm{opt}}(\hat{u}_i, \theta)\right) : \theta \in \Theta \right\}.
\end{equation}
This approach has been effectively employed in multiple domains, including transportation and routing~\citep{bertsimas2015data,zhang2018price,zattoni2025inverse} and healthcare~\citep{chan2014generalized,aswani2019data,chan2022inverse}.

In this work, we consider optimization problems with {\em discrete decision variables}, requiring integer optimization methods for solving the forward problem.  
Current state-of-the-art approaches for mixed-integer linear forward problems are based on cutting-plane algorithms~\citep{wang2009cutting,bodur2022inverse} and stochastic first-order methods~\citep{zattoni2025inverse}.  Because of the computational overhead of solving integer optimization problems, these approaches are less scalable than continuous convex forward problems.
This computational bottleneck motivated us to revisit the assumptions of IO, and propose an alternative approach: {\em one that learns a structured prediction model based on transformers.} 

Our approach is based on an equivalent reformulation of the forward problem~(\ref{eq:FP}) as a reward-based forward problem defined as:
\begin{equation} \label{eq:RF}
\textit{RF}(u, \theta) = \min\limits_{x} \{ R(x,u,\theta) : x \in \mathcal{X}(u) \},
\end{equation}
where $R(x, u, \theta)$ is an abstract reward function that may depend on the contextual information~$u$.  
We approximate $R$ indirectly through a predictive model $M_{\phi}$, 
which for each instance $u \in \mathcal{U}$ defines a probability distribution 
$p_{\phi}(x \mid u)$ over feasible decisions $x \in \mathcal{X}(u)$. Since the predictive model's outputs are high-dimensional, and may display complex interdependencies among its components due to the demand of feasibility, we cast this problem as a structured prediction task \citep{taskar2005learning}. We employ a constraint-reasoning scheme that restricts any given autoregressive generative model to the space of feasible decisions.
We estimate the model parameters $\phi$ by minimizing a supervised loss on historical data:
\begin{equation}\label{eq:ML-loss}
\min\limits_{\phi} \frac{1}{N} \sum_{i=1}^N \ell\big(\hat{x}_i, M_{\phi}(\hat{u}_i) \big),
\end{equation}
where $\ell$ is a loss function measuring the discrepancy between the observed decisions $\hat{x}_i$ and the model-predicted decisions under $M_{\phi}(\hat{u}_i)$.  In doing so, $M_{\phi}$ learns to approximate the mapping from instances to optimal decisions, implicitly capturing both the forward and inverse problems.

Our approach offers several advantages. Firstly, once trained, $M_{\phi}$ produces feasible decisions directly from contextual input, eliminating the need to repeatedly solve an integer optimization problem at inference time. Secondly, by replacing restrictive assumptions that IO places on the functional forms of the objective function and constraints with a flexible predictive architecture, the model can capture complex or nonlinear latent structure, and in our experiments transformer-based models achieve consistently high solution quality together with substantial gains in inference speed over state-of-the-art IO methods.

There are also limitations to our approach. Because the constraint-reasoning module enforces feasibility through a combinatorial state space, identifying admissible decisions may itself be computationally hard. In our case studies, we therefore restrict our attention to monotone constraint systems that we show can be handled efficiently. Also, as a predictive method our model does not provide optimality guarantees, and typically requires more historical data than an IO method. Finally, when the underlying problem is well specified and the data uncorrupted, IO can recover the true objective exactly, whereas our approach may incur unnecessary learning overhead. As a result, the choice between the two paradigms is problem dependent, reflecting complementary strengths.

Another mainstay of data-driven optimization is predict-and-optimize and related decision-focused learning approaches \citep{wilder2019melding, elmachtoub2022smart, mandi2022decision, mandi2023decision}, which learn models tailored to downstream decision quality. However, these are not directly applicable in our setting, because they require additional supervision on the unknown components of the objective function (in our notation, $\theta$) associated with every context $u$ in the dataset, which is unavailable in our setting. Hence, we do not include predict-and-optimize methods as baselines in this study.


\noindent\textbf{Contributions.~} We propose an end-to-end approach for solving discrete optimization problems with unknown components using structured prediction as an alternative to IO. 
Our method combines transformer models, which focus on learning the latent objective function and implicit constraints, with constraint reasoning, representing the known constraints of the problem.  The constraint reasoning module masks out provably infeasible decisions during training and solution generation, thereby ensuring feasibility of the produced solutions.  
We consider IO variants of three classical combinatorial optimization problems: the weighted knapsack problem with unknown objective, the bipartite matching problem with unknown objective, and the single-machine scheduling problem with unknown precedence constraints.
To study the performance of our approach empirically, we examine several forms of learnable latent problem structure, varying along three axes: whether its underlying parameters are observable or hidden, whether their dependence on the data is deterministic or stochastic, and whether they apply globally across all observations or are instance-specific. We also consider multiple functional forms for objective functions (linear and quadratic), and different topologies of precedence constraints.
We compare our method with IO and LSTM-based structured prediction models.  We show that transformers are the only method that consistently achieves strong performance across different settings, including corrupted data and varying size of the training set.
Specifically, our method outperforms IO in almost all cases in terms of solution quality (optimality and feasibility) while generating solutions in a fraction of the time.

\section{Problem Formalization and Latent Structure}\label{sec:problem}


We next present more details on representing the forward and inverse problems as a structured prediction problem.
We discuss different types of latent problem structures, the methodological background on sequence-to-sequence transformer models, and
inductive bias.

\subsection{Problem Definition}
We consider discrete optimization problems with an unknown objective, a set of known constraints and a set of unknown constraints.  From a learning perspective, the unknown components rely not only on the exposed features (via $u$), but also on a set $\mathcal{H}$ of hidden features of the problem. For example, in a routing application the duration of a route can depend on the timing of a nearby sports event, while this information may not be available in $u$.
Because the relative importance of hidden features influences the performance of learning, we want to distinguish these explicit and latent sources of the unknown components in our empirical evaluation.  For this purpose, we introduce a hypothetical extension of the forward problem (\ref{eq:FP}), that relies on the hidden features~$h \in \mathcal{H}$ and separates the known constraints from the unknown constraints:


\begin{equation}\label{eq:RF-v2}
F'(u,h,\theta)
= \min_{x} \left\{
\begin{aligned}
   & f'(x,u,h,\theta) : \\
   & x \in \mathcal{X}(u) \cap \mathcal{X}'(u,h,\theta)
\end{aligned}
\right\},
\end{equation}



\noindent where $f'$ now also depends on $h$, $\mathcal{X}(u)$ represents the known feasible set, 
and $\mathcal{X}'(u,h,\theta)$ represents the unknown feasible set based on~$u$ and~$h$.  Both $f'$ and $\mathcal{X}'$ are hypothetical and assumed to represent the true underlying functional forms of the unknown components.

We next discuss the difference of IO and structured prediction in terms of approximating the hypothetical true functional form.
IO approximates the unknown components of the forward problem $F'$ by solving the inverse problem, resulting in coefficients $\theta$ for $f(u, \theta)$ and $\mathcal{X}(u,\theta)$.  Model misspecification occurs when the algebraic functional form of $f$ and $\mathcal{X}$ do not align with $f'$ and 
$\mathcal{X'}$.
Our structured prediction approach instead estimates the unknown components implicitly through the coefficients $\phi$ of a predictive model $M_{\phi}$ by minimizing the loss function in equation~(\ref{eq:ML-loss}). In this case, misspecification occurs when $M_{\phi}$ does not have the appropriate structure to represent the mapping from the contexts to the optimal decisions. In both cases, the learning process projects the latent structure arising from the hidden features; onto the $(u, \theta)$-space for IO, and onto the $(u, \phi)$ space for structured prediction.  In that sense, IO is more restricted by its reliance on $\theta$ to be defined over the known (algebraic) functional form.  For structured prediction, $\phi$ is defined over the architectural space represented by $M_{\phi}$ which can be of much higher dimension than the functional form of IO, allowing, in principle, for more accurate predictions.
Naturally, the performance and potential misspecification of both approaches also relies on the availability of correct and sufficient training data, which will be discussed next.

\subsection{Types of Latent Problem Structure}

The unknown components of the forward problem may depend on different combinations of the observable contextual features $u$ and the hidden features $h$ introduced above. Because these dependencies determine the form of the true value function $f'$ and the unknown feasible set $X'$, it is useful to distinguish several general types of latent structure that arise in practice. 

In the simplest case, the unknown components depend only on the exposed features $u$. This dependence may be deterministic, following a fixed rule conditioned on $u$, or stochastic, in which case the unknown components are random variables whose distributions depend on $u$. These settings are generally compatible with the algebraic models used in IO, but model misspecification may arise in presence of inconsistent training data.

The unknown components may also depend jointly on $u$ and hidden features $h$. When the mapping is deterministic, the underlying structure is consistent but not fully observable from $u$; when it is stochastic, both hidden information and random variation influence the resulting objective or constraints, introducing additional uncertainty into the optimization problem. In such settings, IO may be structurally misspecified if the hidden features cannot be represented within the assumed algebraic form, whereas a flexible model class for structured prediction may still approximate the induced mapping from $u$ to optimal decisions.

Finally, some latent structures rely on instance-specific encodings of the features, denoted $\varphi(\cdot)$, such as permutations or groupings derived from $u$ or from $(u,h)$. These encodings vary across instances and are not directly observed, requiring the predictive model to infer them implicitly from data. Such structures are difficult to capture with fixed functional forms and tend to favor flexible model architectures capable of representing higher-order interactions or relational patterns.

\begin{table}[tb]
\centering
\caption{Latent problem structures characterizing the forward problem and their applications. Function $\psi(\cdot)$ represents an instance-specific encoding of solution patterns. A detailed description of the case studies can be found in Section \ref{sec:application-problems}.\label{tab:latent-structure-forms}}
{
\resizebox{\textwidth}{!}{%
\begin{tabular}{cp{1in}p{2.6in}p{2.6in}} \toprule
Dependence & Type & Description & Case Study \\ \midrule
$u$ & Deterministic & All unknown components only depend on exposed features $u$ and have a deterministic evaluation. & Linear knapsack problem with unknown objective and single-machine scheduling problem with unknown precedence constraints.\\ \specialrule{0.1pt}{3pt}{3pt}
$u$ & Stochastic & All unknown components only depend on exposed features $u$ and have a stochastic evaluation. & Linear and quadratic bipartite matching problem with unknown objective.\\ \specialrule{0.1pt}{3pt}{3pt}
$(u,h)$ & Deterministic & The unknown components depend jointly on exposed ($u$) and hidden ($h$) features and have a deterministic evaluation. & Quadratic knapsack problem with unknown objective.\\ \specialrule{0.1pt}{3pt}{3pt}
$\psi(u)$ & Instance-specific, deterministic & The unknown components depend on an instance-specific encoding of the exposed features, such as a permutation or ordering derived from sorting $u$. & Heuristic knapsack problem 
 \textit{Alternating (choose~1, skip~1)} and \textit{Alternating (choose~2, skip~1)}.\\ \specialrule{0.1pt}{3pt}{3pt}
$\psi(u,h)$ & Instance-specific, deterministic & The unknown components depend on an instance-specific encoding formed jointly from exposed ($u$) and hidden ($h$) features. & Heuristic knapsack problem \textit{Clustering by group}. \\ \bottomrule
\end{tabular}
}
}
\end{table}

These categories provide the basis for evaluating how different learning approaches handle unknown components of varying complexity.
Table~\ref{tab:latent-structure-forms} lists the latent problem structures that we consider in our application case studies in Section~\ref{sec:application-problems}, identifying for each the associated problem setting. 

\section{Structured Prediction with Transformers}
\label{sec:method}

In this section, we present the structured prediction framework used in our approach. We first introduce the sequence-to-sequence models that learn the mapping from instances to optimal solutions. Then we describe the constraint reasoning mechanism that encodes the known feasible set.  Third, we integrate both components into a generative architecture that produces feasible solutions consistent with the observed latent structure.

\subsection{Sequence-to-Sequence Models}

Our approach requires learning a mapping from an input instance to a structured optimal solution. Sequence-to-sequence models provide a flexible mechanism to approximate such mappings directly from historical data~\citep{neubig2017neural}. They generate an output one element at a time, conditioning each prediction on the full input and on the previously generated partial solution. This autoregressive structure aligns with the constructive nature of many combinatorial solutions, which can be represented as sequences even when the underlying problems are not inherently sequential.

Recurrent neural networks were introduced to model sequential data, but their recursive computation makes it difficult to preserve information over long sequences, as gradients weaken across many recurrent steps. LSTMs (long short-term memory) are a variant of recurrent networks designed to mitigate this limitation by maintaining an internal memory state that carries information across longer spans, making them suitable for sequence-to-sequence architectures \citep{SVL-NIPS2014}. However, LSTMs still propagate information step by step, so the effective distance between two positions in the input grows with the sequence length. Transformer-based models \citep{vaswani2017attention} replace recurrence with self-attention, which directly relates all input elements to one another and shortens the path through which dependencies are learned. Because attention weights are computed over all inputs simultaneously, transformers can represent pairwise and higher-order interactions that arise in combinatorial structures, including those induced by unknown components of the optimization problem such as latent objectives or implicit constraints. Since we represent solutions as sequences, this self-attention mechanism provides a flexible way to learn these latent relationships from historical solutions while generating each decision conditioned on the full instance and the previously constructed partial solution.

However, a sequence-to-sequence model by itself does not guarantee feasibility. The decoder chooses each output based on its predictive distribution, and nothing in the architecture prevents it from selecting elements that violate the combinatorial structure of the problem. As a result, the model may generate sequences that exceed capacity limits, repeat items, omit required elements, or violate matching or scheduling constraints. These feasibility conditions typically restrict the solution space to a small, structured subset of all possible sequences, and the autoregressive mechanism does not exclude infeasible decisions during generation.

For this reason, we use the sequence-to-sequence model primarily to learn the unknown components of the mapping from inputs to optimal solutions, including latent interactions induced by unobserved objectives or implicit constraints. The known combinatorial structure is handled separately. In the next section, we introduce a constraint reasoning module that restricts the decoder's choices at each step in such a way that ensures that the generated sequences satisfy all known constraints while retaining the transformer’s ability to learn latent structure from historical data.


\subsection{Constraint Reasoning}\label{sec:CR}
Recall that the known constraints of the problem define the discrete feasible set $\mathcal{X}(u)$, for known parameters $u$.  Because our approach generates solutions sequentially, and needs to reason over partial solutions, we will assume that the feasible set is described by a deterministic finite automaton (DFA) \citep{HMU2006}.
A DFA is defined as a tuple $(Q, \Sigma, \delta, q_0, F)$, where $Q$ is a finite nonempty set of states, $\Sigma$ is a finite nonempty input alphabet (the set of labels), $\delta : Q \times \Sigma \rightarrow Q$ is the transition function, $q_0$ is the starting state, and $F \subseteq Q$ is the set of accepting (final) states.  
A DFA is a finite-state machine representing a language: a given string is processed by deterministically applying the state transition function based on the labels in the string, starting from $q_0$. If the result is a state in $F$, the string is accepted (i.e., part of the language), and otherwise rejected. To represent strings that are not accepted, there exists a unique non-accepting `sink' state $q_{\bot}$ in $Q$, which represents the maximal set of non-accepting states whose transitions are all directed to itself. Any transition that does not lead to an accepting state is directed to $q_{\bot}$. 
While DFAs are primarily used in the context of automata and languages, they have also been applied to encode the feasible set of combinatorial problems \citep{pesant2004,grammar2011}. They are closely related to the state-action space of dynamic programs, and can provide a compact, structured representation for discrete solution sets.  Any bounded set $\mathcal{X} \subseteq \Sigma^{\leq T}$, for some maximum sequence length $T$, represents a finite language, and therefore a regular language, and a DFA exists to encode it.

To represent the feasible set $\mathcal{X}(u)$ as a DFA, we first assume a fixed but arbitrary ordering on the dimensions of $\mathcal{X}(u)$.  
Each feasible point $x \in \mathcal{X}(u)$ can then be written as a sequence (or string), of discrete labels  $(\sigma_1, \sigma_2, \ldots, \sigma_k)$, where each $\sigma_i \in \Sigma$ belongs to the alphabet of possible decisions.
The DFA $(Q, \Sigma, \delta, q_0, F)$ encodes all feasibility requirements as:
\begin{itemize}
    \item $Q$: the finite set of states, representing all partial configurations or intermediate prefixes;
    \item $\delta(q, \sigma)$: the deterministic rule for transitioning between states by appending decision $\sigma$;
    \item $q_0$: the `empty' starting state (before any decision);
    \item $F$: the set of accepting states (those corresponding to feasible solutions).
\end{itemize}
Then, the set of all accepted strings corresponds exactly to $\mathcal{X}(u)$:
\begin{equation}
\mathcal{X}(u)
= \left\{\, 
\begin{aligned}
   &(\sigma_1,\ldots,\sigma_k) \in \Sigma^k : \\
   & \delta(q_i,\sigma_{i+1}) = q_{i+1}, 0 \leq i \leq k-1,\\
   & q_k \in F, k\leq T
\end{aligned}
\,\right\}.
\end{equation}

We incorporate the DFA as a constraint reasoning tool by defining a masking function that provides the eligible decisions at each stage.

\begin{definition}
Let $(Q, \Sigma, \delta, q_0, F)$ be a DFA. For a state $q \in Q$, we define the {\em mask} $m(q) \subseteq \Sigma$ as the set of labels that do not transition to the sink state $q_{\bot}$, i.e., $m(q) := \big\{ \sigma : \delta(q, \sigma) \in Q \setminus \{ q_{\bot} \} \big\}$.
\end{definition}

\begin{algorithm}[t]
\small
\caption{Sampling a solution from a DFA}
\label{alg:sampling}
\begin{algorithmic}[1]
\Require DFA $(Q, \Sigma, \delta, q_0, F)$ encoding $\mathcal{X}(u)$.
\Ensure Solution $[\sigma_1,\dots,\sigma_k]\in \mathcal{X}(u)$.
\State $\sigma = []$, $q = q_0$ \Comment{Initialize solution and starting state}
\While{true}
  \If{$q \in F$}
    \If{$m(q) = \varnothing$ or $\textsc{RandomTerminate}() = 1$} \Comment{Test for termination}
      \State \Return $\sigma$ \Comment{Return solution}
    \EndIf
  \EndIf
  \State $v \gets \textsc{Sample}(m(q))$
  \Comment{Sample a value $v$ from the mask of $q$}
  \State $\sigma\textsc{.append}(v)$ \Comment{Add $v$ to solution}
  \State $q \gets \delta(q, v)$ \Comment{Transition to next state}
\EndWhile
\end{algorithmic}
\end{algorithm}

Algorithm~\ref{alg:sampling} presents a DFA-based solution construction procedure through random sampling. Starting at the initial state $q_0$, it samples a value from the mask of the state, appends this to the solution vector $\sigma$, and transitions to the next state.  This process is repeated until an accepting state in $F$ is reached. We then return the solution if no more transitions are allowed, or if a termination test with positive probability, e.g., based on a Bernoulli trial, returns 1; otherwise, we continue.

\begin{proposition} \label{prop:sample}
Algorithm~\ref{alg:sampling} returns a solution $\sigma \in \mathcal{X}(u)$ with probability~1 assuming that each feasible transition is sampled with positive probability and $\mathcal{X}(u)$ is non-empty.
\end{proposition}

\begin{proof}[Proof.]
We wish to prove three statements: (a) the algorithm is well defined insofar that $m(q)$ is always non-empty at the invocation of line~8, (b) the algorithm terminates with probability~1, and (c) upon termination, the string produced lies in $\mathcal{X}(u)$. The last of these is immediate since, by definition of the DFA, any finite sequence that reaches a state in $F$ belongs to $\mathcal{X}(u)$, so if Algorithm~\ref{alg:sampling} terminates it returns a feasible solution.

Since $\mathcal{X}(u)$ is not empty, some state in $F$ is reachable from the starting state $q_0$. 
Consider state $q$ at any stage in the algorithm.
By application of the mask, $q$ can never be the non-accepting sink state $q_{\bot}$.
Now assume $q \notin F$.  We claim $m(q)$ is non-empty.
Suppose the opposite.  If $m(q)$ is empty, then by definition all transitions out of $q$ are directed to $q_{\bot}$. By the defining property of the sink state, this implies that $q = q_{\bot}$, contradicting the uniqueness of $q_{\bot}$.
Therefore, throughout the execution of the algorithm, either $q \in F$ or $m(q)$ is non-empty.


Lastly, because every feasible transition from a realizable state is sampled with positive probability, conditioned on the next invocation of the random termination test (line~7) being positive, the induced Markov chain on $Q$ has $F$ as its unique absorbing class and all other states are transient. Therefore, $F$ is reached with probability~1. Because the random termination test (line~7) has positive probability, the algorithm terminates with probability~1 and returns a feasible solution.
\end{proof}

We remark that, although the existence of compact DFAs, that is, ones with a small number of states, is well understood \citep{myhill1957finite,nerode1958linear} and certainly sufficient for an efficient implementation for the proposed algorithm, it is not at all necessary and being able to determine $m(q)$ point-wise, that is, for a given state $q$, is enough. Concretely, although a DFA exists for every bounded set $\mathcal{X} \subseteq \Sigma^{\leq T}$, it may be exponential in $T$.
To execute Algorithm~\ref{alg:sampling} more efficiently, we will assume that the DFA is specified implicitly via a transition rule defined over a set of {\em abstract states}, rather than by explicitly enumerating all DFA states. The concrete DFA states are then those abstract states that are reachable from the initial state under this transition rule.  This is formalized in the following definition.
\begin{definition}
Let $\mathcal{S}$ be a universe of abstract states and $\Sigma$ a finite alphabet.
A \emph{transition rule} over $(\mathcal{S},\Sigma)$ is a mapping 
$\rho : \mathcal{S} \times \Sigma \to \mathcal{S} \cup \{q_{\bot}\}$, where $\rho(q,\sigma)=q_{\bot}$ means the transition is not allowed.~The DFA defined by $\rho$ is $(Q,\Sigma,\delta,q_0,F)$, where $Q$ is the set of
states reachable from $q_0 \in \mathcal{S}$ under~$\rho$ and  $\delta(q,\sigma)=\rho(q,\sigma)$.
\end{definition}
The transition rule specifies the transition semantics over a problem-dependent abstract state space, from which the DFA materializes. For example, let $\mathcal{X}$ be the set of permutations of $\Sigma$.  We can define the transition rules as $(S, i) \rightarrow (S \cup \{i\})$ for all $i \in \Sigma \setminus S$ where the state $S \subseteq \Sigma$ represents the set of previously assigned values.
Enumerating all states and transitions explicitly yields the associated DFA representing $\mathcal{X}$. 
The drawback of using transition rules, however, is that feasibility of a partial solution is no longer guaranteed.  In our applications, we leverage the property of monotonicity of the constraint system to ensure feasibility.
\begin{definition}
A set of feasible sequences $\mathcal{X}$ is {\em monotone} if every prefix of a sequence in $\mathcal{X}$ is feasible, i.e., 
$
(\sigma_1, \dots, \sigma_k) \in \mathcal{X} \Rightarrow 
(\sigma_1, \dots, \sigma_i) \in \mathcal{X}$ for all $i \in \{1, \dots, k-1\}$.
\end{definition}
\begin{proposition} \label{prop:feasibility}
Let $\mathcal{X}(u)$ be a monotone feasible set described by a transition rule $\rho$.  If we 
replace the DFA transitions $\delta$ by $\rho$ in Algorithm~\ref{alg:sampling}, 
the algorithm returns a solution $\sigma \in \mathcal{X}(u)$.
\end{proposition}

\begin{proof}[Proof.]
Since $\mathcal{X}(u)$ is monotone, every prefix of any $\sigma \in \mathcal{X}(u)$ is feasible, so the states reachable from $q_0$ under $\rho$ correspond exactly to feasible prefixes, and $\rho(q,\sigma)=q_{\bot}$ precisely when the extended prefix is infeasible.  Hence the DFA defined by $\rho$ accepts exactly $\mathcal{X}(u)$.  The result then follows from Proposition~\ref{prop:sample}.
\end{proof}

As a consequence of Proposition~\ref{prop:feasibility}, the computational effort of Algorithm~\ref{alg:sampling} is now embedded in the evaluation and processing of the transition rules.  In our applications, this either takes constant time or linear time in the length of the sequence.

Lastly, we note that our DFA-based approach formalizes and generalizes similar masking techniques from the literature, e.g., based on tracking vectors~\citep{kool2018attention,HuaOnt2020} or decision diagrams~\citep{jiang2022constraint}.  It is also related to 
grammar-constrained decoding for structured natural language processing~\citep{GengJP023}
and code generation~\citep{Dong0J23}, but differs in that those methods apply grammars only at decoding time, whereas our DFA is embedded directly into the model and constrains learning and generation throughout.

\subsection{Transformer Architecture with DFA-Based Constraint Reasoning}

We now combine the sequence-to-sequence model and the DFA-based constraint reasoning mechanism introduced in the previous subsections to obtain a generative architecture that realizes the structured prediction model $M_{\phi}$ defined in Section~1. The purpose of this integration is to construct a conditional distribution $p_{\phi}(x \mid u)$ over decisions $x \in \mathcal{X}(u)$ that learns the latent structure of the unknown components while ensuring that all generated solutions satisfy the known constraints.

\begin{figure}[t]
  \centering
  \scalebox{0.9}{\begin{tikzpicture}[
  >=Stealth,
  every node/.style={
    font=\small,
    execute at begin node={\setlength{\baselineskip}{2.5ex}}},
  box/.style={
    rectangle, rounded corners,
    draw=blue!70!black,
    line width=0.9pt,
    fill=white,
    align=center,
    minimum width=3.6cm,
    minimum height=0.9cm,
    inner sep=4pt
  },
  bigbox/.style={
    draw=blue!70!black,
    dashed,
    rounded corners,
    inner sep=6pt
  },
  arrow/.style={->, line width=1pt, draw=blue!70!black}
]

\node[box] (enc) {Embeddings +\\Positional encodings\\[0.75ex](processed by transformer\\encoder layers)};
\node[right=3ex of enc, yshift=5ex] (enc-label) {\bf\textcolor{blue!70!black}{\textsf{ENCODER}}};
\node[box, below=7ex of enc] (dec) {Cross attention\\ + Self attention};
\node[box, below=5ex of dec] (lin) {Linear projection};
\node[box, below=5ex of lin] (mask) {Mask};
\node[box, below=5ex of mask] (soft) {Softmax};
\node[box, below=5ex of soft] (sel) {Select $x_t$ (sample/greedy)};

\draw[arrow] (enc.south) -- node[right,xshift=3pt] {$H^{\text{enc}} \in \mathbb{R}^{n\times d}$} (dec.north);
\draw[arrow] (dec.south) -- node[right] {$h^{\text{dec}} \in \mathbb{R}^{d}$} (lin.north);
\draw[arrow] (lin.south) -- node[right] {$L_t \in \mathbb{R}^{|\Sigma|}$} (mask.north);
\draw[arrow] (mask.south) -- node[right] {$L_t^{\text{masked}} \in \mathbb{R}^{|\Sigma|}$} (soft.north);
\draw[arrow] (soft.south) -- node[right] {$p(x_t \mid u, x_{<t})$} (sel.north);

\node[left=14ex of enc] (u) {};
\draw[arrow] (u.east) -- node[above] {$u \in \mathbb{R}^{n\times k}$} (enc.west) ;

\node[left=14ex of dec] (pref) {};
\draw[arrow] (pref.center) -- node[above]{$x_{<t} \in \Sigma^{t-1}$~~~~~}(dec.west);

\coordinate (xt) at ($(sel.south)+(0,-6ex)$);
\draw[arrow,style={-, line width=1pt, draw=blue!70!black}] (sel.south) -- node[right] {$x_t = \sigma$} (xt);

\draw[arrow,style={-, line width=1pt, draw=blue!70!black}] (xt.west) -| (pref.center);

\node[box, right=14ex of mask] (cr) {Update $q_{t} = \delta(q_{t-1}, x_t)$\\
Construct mask $m(q_t)$};
\node[text=blue!70!black, above=0.5ex of cr.north] (dfa-label) {{\bf\textcolor{blue!70!black}{\textsf{DFA}}}~$(Q,\Sigma,\delta,q_0,F)$};
\node[bigbox, fit=(cr) (dfa-label)] (crbox) {};

\draw[arrow] (xt.east) -| (cr.south);
\draw[arrow] (cr.west) -- node[above] {$m(q_t) \subseteq \Sigma$} (mask.east);

\node[bigbox, fit=(enc) (enc-label)] (encbig) {};

\node[right=23ex of dec, yshift=1ex] (dec-label) {\bf\textcolor{blue!70!black}{\textsf{DECODER}}};

\node[bigbox, fit=(dec) (lin) (mask) (soft) (sel) (xt) (crbox) (dec-label) (pref)] (decbig) {};

\node[below=6ex of xt] (finish) {};
\draw[arrow] (xt) -- node[right,pos=0.75]{$[x_1,x_2, \dots, x_{T}]$} (finish);

\end{tikzpicture}}

  \caption{Transformer encoder-decoder architecture with DFA-guided constraint masking.}
  \label{fig:transformer-arch}

  \caption*{\footnotesize
  \textit{Note.} The encoder maps the instance $u \in \mathbb{R}^{n\times k}$ to contextual embeddings $H^{\rm enc}$.
  At decoder step $t$, the decoder produces logits $L_{t} \in \mathbb{R}^{|\Sigma|}$, which are masked using the feasibility
  vector $m(q_{t})$ returned by the DFA. After softmax, only feasible labels receive nonzero probability.
  The selected label $x_{t}$ (via sampling or greedy decoding) updates both the decoder prefix and the DFA state
  $q_{t} = \delta(q_{t-1}, x_{t})$.}
\end{figure}

The model employs a transformer-based encoder--decoder architecture, augmented with a DFA to restrict the decoder to feasible choices. Figure~\ref{fig:transformer-arch} provides an overview. The input instance $u \in \mathbb{R}^{n\times k}$ is first mapped to contextual embeddings $H^{enc} \in \mathbb{R}^{n\times d}$. Each element of $u$ is embedded into a shared $d$-dimensional representation using learned embeddings for categorical features and feed-forward transformations for continuous ones. Positional encodings impose a consistent ordering for the encoder. A stack of transformer encoder layers applies multi-head self-attention and feed-forward transformations to produce the contextual representations $H^{enc}$.

The decoder generates a sequence $x = (x_{1}, \ldots, x_{T})$ autoregressively. At step $t$, it embeds the prefix $x_{<t}$, applies self-attention over this prefix, and uses cross-attention to the encoder representations $H^{enc}$ to obtain a contextual decoder state.
A linear projection then produces unnormalized scores (or logits) over the alphabet $\Sigma$, which after masking and normalization define the conditional distribution $p_{\phi}(x_{t} \mid u, x_{<t})$.  Together, these conditional distributions specify the autoregressive model $M_{\phi}$.

To enforce feasibility, the decoder is coupled with the DFA $(Q, \Sigma, \delta, q_{0}, F)$ representing 
$\mathcal{X}(u)$. The DFA state evolves with the generated prefix according to $q_{t} = \delta(q_{t-1}, x_{t-1})$. At each state $q_{t}$, the DFA provides the mask of feasible labels
$m(q_{t}) = \big\{\sigma \in \Sigma : \delta(q_{t}, \sigma) \in Q \setminus \{q_{\bot}\}\big\}$.
All labels outside $m(q_{t})$ receive score $-\infty$ before normalization, ensuring $p_{\phi}(\sigma \mid u, x_{<t}) = 0$ for all infeasible labels $\sigma \notin m(q_{t})$.
Consequently, the autoregressive model assigns positive probability only to sequences whose prefixes respect the DFA transitions, i.e., $p_{\phi}(x \mid u) > 0$ only if each $x_{t}\in m(q_{t-1})$, which restricts the support of $M_{\phi}$ to $\mathcal{X}(u)$. The following result formalizes this property.
\begin{proposition}
\label{prop:feasible-generative}
Let $\mathcal{X}(u)$ be encoded by a DFA $(Q,\Sigma,\delta,q_{0},F)$.  
Any sequence generated by the masked autoregressive decoder satisfies $x \in \mathcal{X}(u)$.
\end{proposition}

\begin{proof}[Proof.]
Follows immediately from Proposition~\ref{prop:sample}.
\end{proof}

Given the feasible autoregressive model, training estimates the parameters $\phi$ by minimizing the 
categorical cross-entropy loss.  At each step $t$, the loss is 
\[
\mathcal{L}(p,q)_{t}
= -\sum_{\sigma\in\Sigma} q_{t\sigma}\,\log p_{\phi}(\sigma\mid u,x_{<t}),
\]
where $q_{t\sigma}=1$ for the ground-truth label and $0$ otherwise. 
Because infeasible labels receive zero probability after masking, they contribute no gradient signal. The total loss is $\sum_{t}\mathcal{L}(p,q)_{t}$, and we optimize the model using SOAP~\citep{vyas2024soap}, an optimizer designed for transformer architectures that provides improved training stability.

Once the model is trained, solution generation proceeds through the decoder together with the DFA-guided mask.
Given a new instance $u$, the model computes $p_{\phi}(x_{t}\mid u, x_{<t})$ at each step $t$, applies the feasibility mask $m(q_{t})$, and selects a feasible label to extend the prefix. 
In our experiments we adopt greedy decoding: at each step the label $x_t \in m(q_t)$ with maximum probability is chosen, after which the prefix and DFA state are updated before proceeding to $t+1$.
This process continues until a complete sequence is produced.
Single-sample decoding is also possible but yields inferior performance in our experiments.

Although the DFA-based mask is active in both training and inference, its function differs across these stages.
During training, the mask removes infeasible labels from the softmax, reducing the effective decision space.  During inference, the mask fully determines the set of admissible next labels, ensuring that every generated sequence satisfies the known constraints encoded by $\mathcal{X}(u)$.
These elements, together with the encoder–decoder model, define the predictive framework used in subsequent sections.
The next subsection examines the inductive biases associated with this representation and learning scheme.

\subsection{Inductive Bias Considerations} 
In machine learning, an inductive bias is the set of assumptions that guide a model’s ability to generalize beyond the training data \citep{bengio2013representation}. These assumptions determine how the model selects among competing hypotheses to explain the data. Inductive biases can arise from the architecture itself (architectural bias), from how the input and output are represented (representation bias), or from the learning dynamics of the training procedure (algorithmic bias), such as the choice of optimizer and related hyperparameters. In our setting, these biases shape how well models capture latent problem structures. We account for these biases when determining methodological aspects such as input/solution encoding, architecture, optimizer, and hyperparameters.

\section{Application to Three Combinatorial Problems}
\label{sec:application-problems}

We next instantiate our structured prediction framework for three combinatorial optimization problems with unknown components: the knapsack problem with an unknown reward function, the bipartite matching problem with an unknown objective function, and the single-machine scheduling problem with release times and unknown precedence constraints. These problems span increasing levels of structural complexity, allowing us to examine how the model handles unknown objectives, instance-specific latent structure, and hidden feasibility constraints. Each problem requires distinct methodological choices regarding input representation, output encoding, and DFA construction, which we describe in the subsections below.

\subsection{Knapsack Problem with an Unknown Reward Function}
\label{ss:kp}

We study knapsack problems with an unknown reward function. The classical 0–1 knapsack with a linear objective is NP-hard but admits a pseudo-polynomial-time dynamic program, whereas the quadratic knapsack is strongly NP-hard~\citep{kellerer2004multidimensional}.

Let $U$ be a set of $n$ elements, where each element $j \in U$ is characterized by a weight $w_{j} \in \mathbb{R}^{+}$ and an individual reward $c_{j}^{1} \in \mathbb{R}^{+}$. In addition, we define quadratic rewards $c_{jk}^{2} \in \mathbb{R}$ for elements $j, k \in U$.  Each element $j$ belongs to a group with label $g_j \in G$, where $G$ denotes the set of possible groups. Let $B \in \mathbb{R}^{+}$ be the maximum capacity of the knapsack.  
As decision variables, we define the binary vector $y \in \{0,1\}^{|U|}$, where $y_{j} = 1$ indicates the inclusion of element $j \in U$ in the knapsack. 
In this setting, both the functional form of the objective and the parameters that specify it are unknown to the decision maker. Note that the vector $c = (c^{1}, c^{2})$ is unknown to the decision maker and determines only the objective function, and not the feasible set. We consider three variations of the knapsack problem, all subject to the weight capacity constraint. The \textit{linear knapsack} maximizes the sum of individual rewards $c^{1}$:
\begin{equation}\label{eq:lin-knap}
    \textbf{KP-L}(c^{1},u) = 
\max_{\, y\in\{0,1\}^n}\; \Bigl\{\sum_{j=1}^{n} c^{1}_{j} y_{j}
\;:\; w^\top  y\leq B\Bigr\}
\end{equation}

The \textit{quadratic knapsack} extends it by incorporating pairwise rewards $\mathbf{c}^{2}$: 
\begin{equation}
\label{eq:quad-knap}
\begin{alignedat}{2}
\textbf{KP-Q}(c,u) \;=\;
& \max_{y \in \{0,1\}^{n}} \; & &
   \sum_{j=1}^{n} c^{1}_{j} y_{j}
   + \sum_{\substack{j,k=1\\ j\neq k}}^{n} c^{2}_{jk} y_{j} y_{k} \\
& \text{s.t.} \; & & w^{\top} y \le B
\end{alignedat}
\end{equation}

The {\em heuristic knapsack} reflects a decision maker with bounded rationality that uses a heuristic rule, reflected in a function $f_{\mathrm heur}$,  to identify solutions:
\begin{equation}
\label{eq:heur-knap}
\begin{alignedat}{2}
\textbf{KP-HEUR}(u) \;=\;
& \max_{y \in \{0,1\}^{n}} \; & & f_{\mathrm{heur}}(y) \\
& \text{s.t.}\; & & w^{\top}y \le B
\end{alignedat}
\end{equation}


\noindent \textbf{Data Representation and Other Methodological Choices.~} An instance is given by a subset of elements $U = \{e_{1},..., e_{n}\}$, a list of associated weights $w_j$ for each element $j \in U$, and a capacity $B$.  An associated solution is a subset of elements $S \subseteq U$, say $S = \{s_1, \dots, s_k\}$. We encode these as the vectors $u_{1} = (e_1, \dots, e_n)$, $u_{2} = (w_1, \dots, w_n)$, $u_{3} = (B)$, and
$x = (s_1, \dots, s_k)$. Here, the hidden features are the groups, $h = (g_{1},..., g_{n})$.

Since the knapsack problem is not inherently sequential, we introduce a preprocessing step that imposes a sorting rule $\rho_{\text{knap}}$ on both input and output sequences. Specifically, for each element, we compute its empirical frequency of inclusion in solutions, subject to its inclusion in the problem instance, and sort elements in descending order of these frequencies. This ordering induces a permutation $\pi_{\text{knap}}$ that provides a consistent sequential representation across instances. Then, the input context vector corresponds to $u = u_{1}^{\pi_{\text{knap}}}\,\|\,u_{2}^{\pi_{\text{knap}}}\,\|\,u_{3}$ and its solution $x^{\pi_{\text{knap}}}$.



\medskip 

\noindent\textbf{Constraint Reasoning.~} The feasible set for the knapsack problem consists of subsets of items whose total weight does not exceed the capacity $B$. We represent solutions as sequences in which items may be selected in arbitrary order, without repetition. We define the associated DFA as follows. Each state is a pair $(S,r)$ where $S \subseteq U$ denotes the set of items already selected and and $r \in [0,B]$ is the remaining capacity. The alphabet is $\Sigma =  U$. We define the initial state as $q_{0} = (\varnothing, B)$. The transition rule is defined as: $\rho((S,r), j) = (S \cup \{j\}, r - w_{j})$, for all $j \in U\setminus S$  if $r - w_j \geq 0$. Selection may terminate at any state. Monotonicity holds because any partial selection that respects the remaining capacity is feasible. Note that although this DFA has exponentially many states, we only materialize a single path at each iteration. During decoding, feasibility is enforced by maintaining the set of previously selected items and the remaining capacity, and masking actions that would violate either constraint. As a result, each transition can be evaluated and processed in constant time.


\medskip

\noindent\textbf{Instance and Solution Generation.~} We construct a universe of $n$ elements and generate weights and capacities using the procedure proposed by~\cite{pisinger2005hard}, which ensures that the instances are challenging and diverse. Each instance consists of a subset of elements and a corresponding capacity. For more details, see Appendix~\ref{sec:appendix-instances-knap}.

To define the reward functions, we introduce linear and quadratic forms, presented in Table~\ref{tab:functional-forms}. Linear rewards depend solely on element weights, while quadratic rewards capture pairwise interactions between elements, based on group labels. In particular, the quadratic value function reflects a practical application in marketing strategies in which elements represent customers. Inspired by a niche marketing strategy, it divides elements into groups corresponding to market segments, penalizing selections across different groups while rewarding selections within the same group.



\begin{table}[tb]
\centering
\small
\caption{Linear and quadratic reward value functions\label{tab:functional-forms}}
\begin{tabular}{@{}llp{6.5cm}@{}}
\hline
\textbf{Reward type} & \textbf{Name} & \textbf{Function} \\
\hline
\multirow{2}{*}{Linear reward $c^{1}_{j}$}
    & Inverse proportionality & $c^{1}_{j} = \frac{1}{w_{j}}$ \\
\cline{2-3}
    & Logarithmic             & $c^{1}_{j} = \log(w_{j})$ \\
\hline
Pairwise reward $c^{2}_{jk}$ 
    & Including negative &
        $c^{2}_{jk} = 0 \quad \forall j,k: j = k$ \newline
        $c^{2}_{jk} = q_1(c^{1}_{j} + c^{1}_{k}) + q_2 \quad \forall j,k: g_{j} = g_{k}$ \newline
        $c^{2}_{jk} = -q_3(c^{1}_{j} + c^{1}_{k}) \quad \forall j,k: g_{j} \neq g_{k}$ \\
\hline
\end{tabular}
\caption*{\footnotesize $q_1, q_2$, and $q_3$ are constants.}
\end{table}


We also study heuristic reward functions that induce instance-specific solution patterns not captured by linear or quadratic value functions.  
These correspond to instantiations of the instance-specific encodings $\psi(u)$ and $\psi(u,h)$ introduced in Table~\ref{tab:latent-structure-forms}.  Specifically, we use the following heuristic value functions $f_{\mathrm heur}$ with increasing cluster patterns, each of which is applied to the set of elements $U$ in an instance:
\begin{itemize}
    \item \textit{Alternating (choose 1, skip 1)}: We select every other element in ascending index order until the capacity is reached. Such instances test the model's ability to learn simple positional patterns.
    \item \textit{Alternating (choose 2, skip 1)}: Similar to the previous rule, but we select two consecutive elements, skipping one element after every pair. Such instances stress the importance of learning two-element positional clusters, in addition to the positional pattern.
    \item \textit{Clustering by group}: For each instance, based on the elements group labels, we select the group with the largest presence, sort its elements by their index, and add them to the knapsack until capacity is met. These test the ability to learn group-based clustering.
\end{itemize}


After generating the instances, we obtain their corresponding solutions using two approaches. 
For the linear and quadratic knapsack problems, we solve the associated optimization problem for each instance, i.e., given $c = (c^{1}, c^{2})$, we solve the optimization problems defined by \textbf{KP-L}$(c^{1}, u)$ and \textbf{KP-Q}$(c, u)$, respectively. Linear instances are solved quickly in practice. In contrast, some of the quadratic instances are difficult to solve optimally, in which case we use the best solution found within the time limit (10 min).  If no solution is found, we ignore the instance.
For the heuristic knapsack problem, we apply the heuristic rule to obtain the solution associated to each instance.

\subsection{Bipartite Matching Problem with an Unknown Objective Function}

We study the bipartite matching problem with an unknown objective. The linear bipartite matching is solvable in polynomial time, for example, via the Hungarian algorithm. The quadratic bipartite matching is in general NP-hard by reduction from 0-1 quadratic programming, which is NP-hard~\citep{caprara2008constrained}.

The bipartite matching problem arises naturally in a wide range of applications that require assigning elements from one set to another. It considers a bipartite graph $G = (V_{1}, V_{2}, E)$, where $V_{1}$ and $V_{2}$ correspond to the set of vertices on the left-hand and right-hand side, respectively. $E \subseteq V_{1} \times V_{2}$ corresponds to the set of edges, where each edge is associated with a non-negative reward $c^{1}_{jk} \in \mathbb{R}^{+}$ \citep{lovasz2009matching}. Additionally, we consider quadratic rewards $c^{2}_{(jk),(lm)} \in \mathbb{R}$. As decision variables, we define binary variables $y_{jk} \in \{0,1\}$ for each edge $(j,k) \in E$, where $y_{jk} = 1$ indicates the inclusion of edge $(j,k)$ in the matching. Let $\delta(j)$ be the set of edges incident to vertex $j$.

Similarly to the knapsack problem, both the functional form and the parameters that specify it are unknown to the decision maker. Here, $c = (c^{1}, c^{2})$ is unknown to the decision maker and determines only the objective function. We consider two variations of the bipartite matching problem,
both subject to the standard matching constraints. The \textit{linear bipartite matching} maximizes the
sum of individual rewards $c^{1}$:
\begin{equation}
\label{eq:lin-matching}
\begin{alignedat}{2}
\textbf{BM-L}(c^{1}, u) \;=\;
& \max_{y \in \{0,1\}^{|E|}} \; & & \sum_{e\in E} c^{1}_{e}\, y_{e} \\[2pt]
& \text{s.t.}\; & &
  \sum_{e\in \delta(k)} y_{e} \le 1\quad \forall k\in V_{2} \\
& & &
  \sum_{e\in \delta(j)} y_{e} \le 1\quad \forall j\in V_{1}
\end{alignedat}
\end{equation}

The \textit{quadratic bipartite matching} extends it by incorporating pairwise rewards $c^{2}$:
\begin{equation}
\label{eq:quad-matching}
\begin{aligned}
\textbf{BM-Q}(c,u)
= \max_{y \in \{0,1\}^{|E|}} \;&
   \sum_{e\in E} c^{1}_{e}\, y_{e}  \\
& {}+ \sum_{\{e,f\}\subseteq E} c^{2}_{ef}\, y_{e} y_{f} \\[2pt]
\text{s.t.}\quad &
   \sum_{e\in \delta(k)} y_{e} \le 1 \quad \forall k\in V_{2}, \\[2pt]
&
   \sum_{e\in \delta(j)} y_{e} \le 1 \quad \forall j\in V_{1}.
\end{aligned}
\end{equation}

\noindent \textbf{Data Representation and Other Methodological Choices.~} Each node $j \in V_{1} \cup V_{2}$ belongs to a group with label $g_j \in G$, where $G = \{1,...,m\}$ denotes the set of possible groups. These groups are fixed across all instances. Each instance is defined by a subset of nodes on the left-hand side $S_{1} \subseteq V_{1}$, a subset of nodes on the right-hand side $S_{2} \subseteq V_{2}$, and the corresponding set of edges $S_{E} = \{(j,k): j \in S_{1}, k \in S_{2}\}$.

We adopt the same encoding scheme as in the knapsack problem. Each instance is defined by a subset of edges $U = \{e_{1},..., e_{r}\}$, along with two vectors specifying the group labels of the left-hand and right-hand nodes associated with each edge, respectively. An associated solution corresponds to a subset $S \subseteq U$ of edges. We encode these as the vectors $u_{1} = (e_{1},..., e_{r})$, $u_{2} = (g_{1},..., g_{r})$, $u_{3} = (g_{1},..., g_{r})$, and $x = (e_{1},..., e_{l})$, respectively.

This optimization problem is also not inherently sequential, so we impose an ordering $\rho_{\text{match}}$, similar to the one in the knapsack application, that induces a permutation $\pi_{\text{match}}$ on the input and output data. For each edge, we compute its empirical frequency of occurrence in the solution, conditional on being present in the instance, and sort all edges in descending order of these frequencies. Then, we have $u = u_{1}^{\pi_{\text{match}}}\,\|\,u_{2}^{\pi_{\text{match}}}\,\|\,u_{3}^{\pi_{\text{match}}}$ and $x^{\pi_{\text{match}}}$.

\medskip

\noindent\textbf{Constraint Reasoning.~} The feasible set for the matching problem consists of subsets of edges forming a matching, i.e., no two selected edges share an endpoint. We represent solutions as sequences in which edges may be selected in arbitrary order, subject to matching constraints. We define the associated DFA as follows. Each state is a subset $R \subseteq U$ representing the remaining eligible edges. The alphabet is $\Sigma =  U$. We define the initial state as $q_{0} = U$. For each edge $i = (u,v) \in U$, we define its neighborhood as the set of edges
$N_i := \{ (u', v') : (u', v') \in U \wedge ((u' = u) \vee (v' = v)) \}$. The transition rule is defined as: $\rho(R, j) = R \setminus N_{j}$ for all $j \in R$. The set of terminal states is defined as $F = \{ R : R \subseteq U \}$, allowing to terminate at any state. Monotonicity holds because any partial set of non-conflicting edges is feasible. As in the previous application, although this DFA has exponentially many states, we only materialize a single path at each iteration. During decoding, feasibility is enforced by maintaining the set of selected edges and masking all conflicting edges. Observe that each transition function can be evaluated and processed in $O(|U|)$ time.


\medskip

\noindent\textbf{Instance and Solution Generation.~} We consider bipartite matching instances defined on a complete bipartite graph with left-side nodes $V_{1}$ and right-side nodes $V_{2}$. Each instance corresponds to a subset of feasible edges $S_{E} \subseteq E$. Rewards are based on group-dependent interactions (see Table~\ref{tab:group_rewards-matching}). Linear rewards depend on pairwise group affinities between endpoints, while quadratic rewards include pairwise interactions (see Appendix~\ref{sec:appendix-instances-match} for more details). 

\begin{table}[tb]
\centering
\small
\caption{Linear and quadratic rewards for the bipartite matching problem.\label{tab:group_rewards-matching}}
\begin{tabular}{llp{7.5cm}}
\hline
\textbf{Reward type} & \textbf{Name} & \textbf{Function} \\
\hline
Linear reward $c_{jk}^{1}$ for edge $(j,k)$ & Group-based & 
$c_{jk}^{1} \sim \mathcal{U}(a_d, b_d)$, where $d = |g_{j} - g_{k}|$, and $(a_d, b_d)$ defines the reward range for group distance $d$. \\
\hline
\begin{tabular}[l]{@{}r@{}}Pairwise reward $c_{(jk),(lm)}^{2}$\\for edges $(j,k)$ and $(l,m)$\end{tabular} & Diversity-based & 
$c_{(jk),(lm)}^2 =
\begin{cases}
-\gamma, & \text{if } g_j = g_k = g_l = g_m\\
+\gamma, & \text{if } g_l = g_m \text{ and exactly one of } g_j, g_k = g_l \\
0, & \text{otherwise}
\end{cases}$ \\
\hline
\end{tabular}
\caption*{\footnotesize $\mathcal{U}$ denotes a uniform distribution. Each job $j$ has an integer group label $g_{j} \in \mathbb{Z}^{+}$.}
\end{table}


For each instance, we obtain their corresponding optimal solution by solving its associated optimization problem. That is, given $c = (c^{1}, c^{2})$, we solve the forward problem, which, depending on the underlying objective function, corresponds to $\textbf{BM-L}(c^{1}, u)$ or $\textbf{BM-Q}(c, u)$.

\subsection{Single-Machine Scheduling Problem with Release Times and Unknown Precedence Constraints}

Next we consider a scheduling problem in which the unknown components lie within the constraints. In particular, we study the version that minimizes the sum of completion times, which together with arbitrary precedence constraints, is strongly NP-hard \citep{lenstra1977complexity}.

Let $U$ be a set of $n$ jobs, where each job $j \in U$ is characterized by a group label $g_{j} \in G$, with $G$ denoting the set of groups, and a processing time $p_{j} \in \mathbb{Z}^{+}$. These values are fixed across all instances in our dataset. Each instance consists of a subset $S \subseteq U$ of jobs. For each instance, each job $j \in S$ is characterized by a release time $r_{j} \in \mathbb{Z}^{+}$.

We define the decision variables as follows. Let $z \in \{0,1\}^{|S|\cdot |S|}$ be a binary matrix where $z_{jk} = 1$ indicates that job $j$ is scheduled before job $k$. For each job $j \in S$, let $s_{j} \in \mathbb{Z}^{+}_{0}$ denote its starting time, and $w_{j} \in \mathbb{Z}^{+}_{0}$ its completion time. The objective is to minimize the average completion time, and preemption is not allowed. Additionally, we incorporate group-based precedence constraints denoted in $\mathcal{Q}$, that are fixed across all instances but unknown to the decision maker. Let $M = \sum_{j \in S} r_{j} + p_{j}$. The corresponding optimization problem is presented in equation~(\ref{eq:single-machine}).
\begin{equation}
\label{eq:single-machine}
\begin{alignedat}{2}
\textbf{SMS}(\mathcal Q, u)
&= \min_{s \in \mathbb{R}^{+},w \in \mathbb{R}^{+},z \in \{0,1\}^{|S|\cdot |S|}}\sum_{j\in S} w_j  && \\[2pt]
\forall j\in S:\;      &\quad w_j = s_j + p_j \\[-2pt]
\forall j\in S:\;      &\quad s_j \ge r_j \\
\forall j\ne k\in S:\; &\quad s_k \ge w_j - M(1- z_{jk}) \\
\forall j\ne k\in S:\; &\quad s_j \ge w_k - M z_{jk} \\
\forall j\ne k\in S:\; &\quad z_{jk} + z_{kj} = 1 \\
\forall j\in S:\;      &\quad z_{jj} = 0 \\
                       &\quad z\in\mathcal Q
\end{alignedat}
\end{equation}

\noindent \textbf{Data Representation and Other Methodological Choices.~} Since release times vary across instances, we do not impose a fixed input order and therefore omit positional encodings in the transformer model input. On the output side, the model predicts a job ordering that corresponds to the scheduling solution, making output position inherently meaningful.

An instance is given by a subset of jobs $S = \{e_1,\dots,e_k\}$, a list of associated processing times $p_j$, a list of their release times $r_{j}$, and a list of their groups $g_{j}$ for each job $j \in S$. An associated solution is an ordered subset of jobs in $S$. We encode these as the vectors $u_{1} = (e_{1},...,e_{k})$, $u_{2} = ( p_{1},...,p_{k})$, $u_{3} = (r_{1},...,r_{k})$, $u_{4} = (g_{1},..., g_{k})$, and $x = (s_{1},...,s_{k})$, respectively. The input context vector is $u = u_{1}\,\|\,u_{2}\,\|\,u_{3}\,\|\,u_{4}$. Note we do not encode precedence constraints as we assume they are unknown to the decision maker. Given that a schedule can be represented as a permutation of jobs, our constraint reasoning maintains a tracking vector that enforces only permutation feasibility during decoding. Release time constraints are enforced \textit{ex post} when computing performance metrics.

\medskip

\noindent\textbf{Constraint Reasoning.~} The feasible set for the scheduling problem consists of vectors of length $k$ representing a permutation of jobs in $U$.
We define the associated DFA as follows.  Each state is defined by a set $R \subseteq U$ representing the remaining eligible jobs.  The alphabet is $\Sigma =U$.  We define the initial state as $q_0 = U$.
The transition rule is defined as:
$\rho( R, j ) = R \setminus \{j\}$ for all $R \subseteq U$ and $j \in R$.
The set of terminal states is defined as $F = \{ \varnothing \}$, ensuring that all jobs are scheduled in the permutation.  Monotonicity holds because any partial permutation of jobs is feasible.  As before, we only materialize a single path in the DFA at each iteration.  Observe that each transition function can be evaluated and processed in constant time.

\medskip

\noindent\textbf{Instance and Solution Generation.~} We consider scheduling instances defined over a universe of jobs, each characterized by a group label, a processing time, and a release time. Release times and processing times are generated using standard benchmark procedures (see Appendix~\ref{sec:appendix-instances-scheduling} for details). We also introduce group-based precedence constraints, represented as a directed acyclic graph over job groups, and shared across all instances. After generating the instances, we obtain their corresponding solutions by solving \textbf{SMS}$(\mathcal{Q}, u)$ to optimality.



\section{Algorithmic Benchmarks}
\label{sec:algorithmic-benchmarks}

We evaluate our proposed method against two competing approaches: a state-of-the-art IO approach and an alternative ML-based approach. 

\subsection{Inverse Optimization Approach}
\label{ss:io-approach}

IO seeks to infer objective-function parameters that make the historical decisions (approximately) optimal for the forward problem introduced in Section~\ref{sec:Intro}. Rather than learning the solution mapping directly from data, IO assumes a parametric form for the unknown objective $f_{\mathrm{inv}}$ and uses the input--solution pairs $\{(\hat u_i,\hat x_i)\}_{i=1}^N$ to estimate its parameters. This requires solving a backward problem to determine the parameters and repeatedly solving the associated forward problem within the algorithmic procedure.

IO is a well-established research area in operations research, see \citet{heuberger2004inverse} and \citet{chan2025inverse} for surveys. Early work includes inverse shortest path and network flow problems \citep{ahuja2001inverse}, as well as approaches based on decomposition techniques such as column generation to handle large feasible sets \citep{yang1999two}. The computational complexity of IO is not explained solely by the difficulty of the forward problem: even when the forward combinatorial problem is solvable in polynomial time, the corresponding inverse problem can be NP-hard \citep{cai1999complexity}. As a result, practical IO methods for discrete problems typically rely on iterative cutting-plane or constraint generation schemes that repeatedly solve the forward problem \citep{bodur2022inverse, wang2009cutting}.

We adopt the multipoint trust-region cutting-plane method of \citet{bodur2022inverse}. Its data-driven backward model, \textbf{IO}$(\lambda,\{\hat u_i,\hat x_i\}_{i=1}^N)$, is
\begin{align}
\min_{\theta, \theta_{i}} \quad & 
\|\theta^0 - \theta\|_1 
+ \lambda \sum_{i=1}^{N} \|\theta_{i} - \theta\|_1 \notag \\
\text{s.t.} \quad 
& f_{\text{inv}}(\hat{x_{i}}, \hat{u_{i}}, \theta_{i}) 
\geq f_{\text{inv}}(x, \hat{u_{i}}, \theta_{i}) 
\notag \\
& \forall i = 1, \dots, N,\;
\forall x \in \mathcal{E}(\mathcal{X}(\hat{u_{i}}, \theta_{i})).
\label{eq:inverse-optimization}
\end{align}
where $\mathcal{E}(\mathcal{X}(\hat u_i,\theta_i))$ denotes the extreme points of the feasible region in the forward problem. The algorithm alternates between a master problem, which is a relaxation of \textbf{IO}$(\lambda,\{\hat u_i,\hat x_i\}_{i=1}^N)$, and a subproblem for each observation that solves the forward problem under candidate parameters~$\theta_i$. An initial guess $\theta^0$ is required, and performance depends on the assumed functional form of $f_{\mathrm{inv}}$ and the complexity of the forward problems.

In our experiments, this IO method is used as a baseline for comparison. All problem-specific forward models, parametric objective families, initialization schemes, and objective variants are provided in Appendix~\ref{sec:appendix-algo-bench}.

\subsection{LSTM-based Approach}

We also compare our method with a sequence-to-sequence LSTM model implemented within the same structured prediction framework. In this baseline, the transformer encoder–decoder is replaced by an LSTM encoder–decoder, while all other components of the pipeline remain identical: the input encoding, output representation, autoregressive decoding mechanism, and DFA-based constraint reasoning follow exactly the same design. The LSTM model is trained using the same cross-entropy loss on the observed solution sequences. Because LSTMs impose a different architectural inductive bias than transformers---relying on recurrent state updates rather than self-attention---they provide a natural and meaningful comparator for assessing how model architecture influences the ability to capture latent structure. Architectural and hyperparameter details are reported in Appendix~\ref{sec:appendix-algo-bench}.

\section{Performance Metrics and Experimental Configurations}
\label{sec:experimental-configuration}

We now describe the performance metrics and experimental configurations used to evaluate the methods above. In addition to the IO and LSTM baselines, we include simple problem-specific heuristics (e.g., sorting or greedy rules) to contextualize performance for each application; their definitions are provided in Appendix~\ref{sec:exp-conf}.

Because the unknown components of the objective or constraints are not available to any of the benchmarked methods, true optimality certificates cannot be computed. We therefore evaluate the performance relative to the optimal solutions in the validation set, which are computed using the full problem specification.
Our primary metric is the optimality gap relative to these reference solutions. We also report feasibility percentage and edit distance, depending on the application. 

Across the three applications, we consider four experimental variations:
(i) Models trained using the entire dataset; (ii) Models trained on only 10\% of the training data to assess sensitivity to sample size; (iii) Models trained on data corrupted by 5\% of suboptimal solutions to test robustness against contamination; and (iv) The omission of Constraint reasoning during training, which is applied only during inference. The specific details for each application can be found in Appendix~\ref{sec:exp-conf}.

\section{Experimental Setup and Evaluation}
\label{sec:experiments}

We next present the empirical results for the three application problems introduced in Section~\ref{sec:application-problems}.  The experiments are carried out on a system with an Intel Core i7 processor, 32 GB RAM, and an NVIDIA RTX 4090 GPU with 24 GB VRAM.  We implement all ML models using Python and PyTorch. Additional implementation details can be found in Appendix~\ref{sec:implementation}.

\subsection{Knapsack Problem with an Unknown Reward Function}

We first present results for the knapsack problem under the linear, quadratic, and heuristic settings introduced in Section~\ref{ss:kp}.  Optimal solutions for all knapsack instances are computed using Gurobi with a 12-hour time limit per instance. We consider instances drawn from a universe of 100 elements, generated using varying item sizes and capacities. See Appendix~\ref{sec:knap} for full details.

\smallskip

\noindent\textbf{Linear Knapsack.~} 
Figure~\ref{fig:gaps-linear-knapsack} presents the optimality gaps and Table~\ref{tab:linear-knap-0gaps} reports the percentage of optimally solved instances. As expected, we observe that the IO method performs very well on linear instances. Likewise, transformers perform competitively, generally achieving the highest percentage of correctly solved instances among all methods and sometimes recovering more instances than the \textit{Omniscient greedy} heuristic. LSTM-based models consistently underperform transformers.

Enforcing constraint reasoning during training consistently improves performance, with the effect being more pronounced for transformers than for LSTMs. Under corrupted data, the IO method shows the largest performance decline, performing even worse than the \textit{Random} baseline. For smaller training sets, LSTM-based models are the only method to experience a significant performance drop when learning from fewer observations.
These observations reflect the fact that the linear objective is easy to recover and aligns well with the IO formulation, whereas transformer models learn this structure reliably from data.

\begin{figure}[tb]
  \centering

  \includegraphics[width=\textwidth]{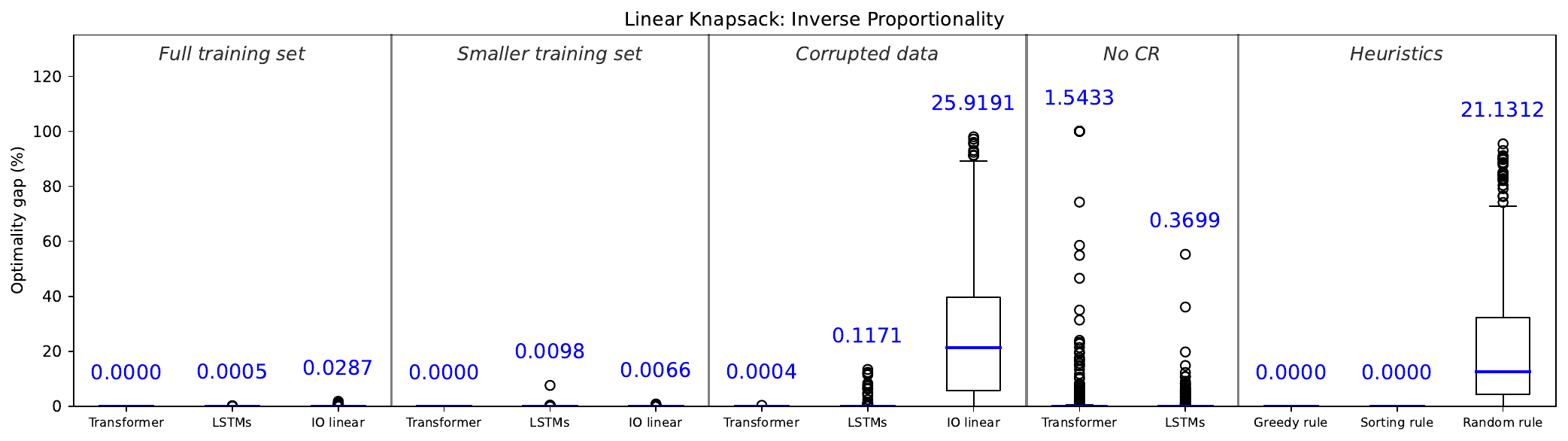}

  \vspace{0.3em}

  \includegraphics[width=\textwidth]{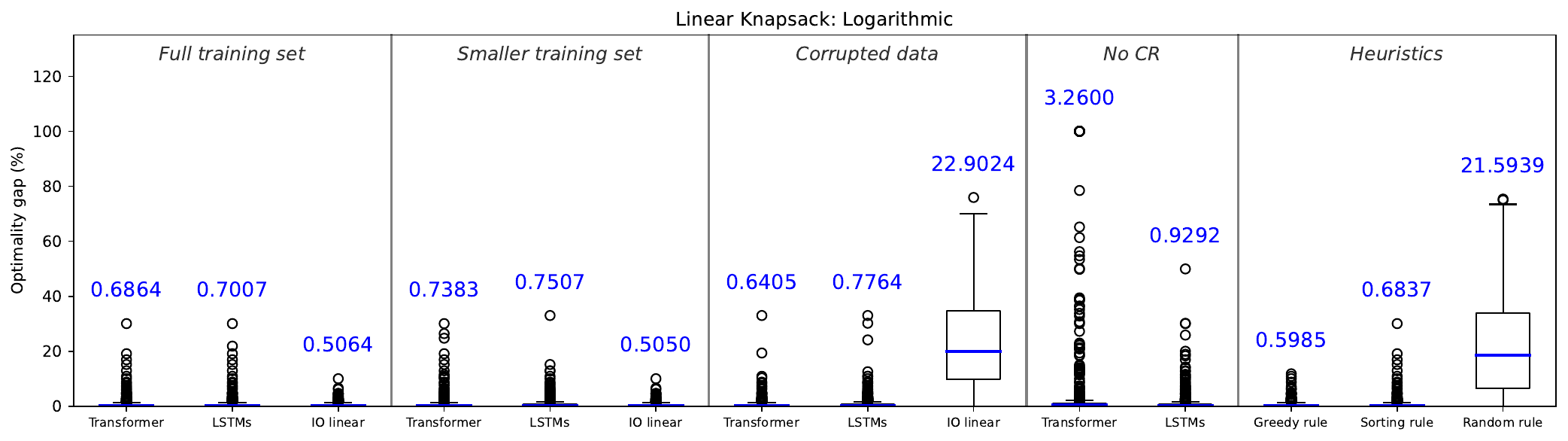}

  \caption{Optimality gaps (\%) for the linear instances of the knapsack problem.}
  \label{fig:gaps-linear-knapsack}

  \vspace{0.25em}
  {\footnotesize\emph{Note.} Means are displayed on top of the boxes.}
\end{figure}

\begin{table}[tb]
\centering
\small
\caption{Percentage (\%) of solutions solved to optimality for the linear instances of the knapsack problem.\label{tab:linear-knap-0gaps}}
\resizebox{\textwidth}{!}{%
\begin{tabular}{llrrrrrr}
\hline
\textbf{Instance type} & \textbf{Experiment} & \textbf{Transformers} & \textbf{LSTMs} & \textbf{\begin{tabular}[c]{@{}l@{}}Linear\\ IO\end{tabular}} & \textbf{\begin{tabular}[c]{@{}l@{}}Omniscient\\ greedy\end{tabular}} & \textbf{\begin{tabular}[c]{@{}l@{}}Sorting\\ rule\end{tabular}} & \textbf{\begin{tabular}[c]{@{}l@{}}Random\\ rule\end{tabular}} \\
\hline
\multirow{4}{*}{\begin{tabular}[c]{@{}l@{}}Linear inversely\\proportional\end{tabular}} 
    & Full training set     & 100.00 & 99.60 & 51.64 & \multirow{4}{*}{100.00} & \multirow{4}{*}{100.00} & \multirow{4}{*}{2.20} \\
    & Smaller training set  & 100.00 & 98.40 & 71.52 &                        &                        &                        \\
    & Corrupted data        & 99.90  & 96.10 & 2.00  &                        &                        &                        \\
    & No CR during training & 72.93  & 90.31 & N/A   &                        &                        &                        \\
\hline
\multirow{4}{*}{\begin{tabular}[c]{@{}l@{}}Linear\\ logarithmic\end{tabular}} 
    & Full training set     & 7.29 & 7.19 & 7.99 & \multirow{4}{*}{6.39} & \multirow{4}{*}{6.89} & \multirow{4}{*}{2.10} \\
    & Smaller training set  & 7.39 & 7.19 & 7.89 &                     &                   &                     \\
    & Corrupted data        & 7.39 & 7.49 & 2.20 &                     &                   &                     \\
    & No CR during training & 4.40 & 6.79 & N/A  &                     &                   &                     \\
\hline
\end{tabular}}
\caption*{\footnotesize CR denotes constraint reasoning and N/A indicates not applicable.}
\end{table}

\medskip

\noindent\textbf{Quadratic Knapsack.~} 
Quadratic instances introduce pairwise interactions that make the problem more challenging than the linear case.  The resulting optimality gaps are shown in Figure~\ref{fig:gaps-quadratic-knapsack}. We observe more pronounced differences in performance between transformers and LSTMs, consistent with the observation in Section \ref{sec:method} that LSTM-based models struggle to capture long complex dependencies.

\begin{figure}[tb]
  \centering
  \includegraphics[width=\textwidth]{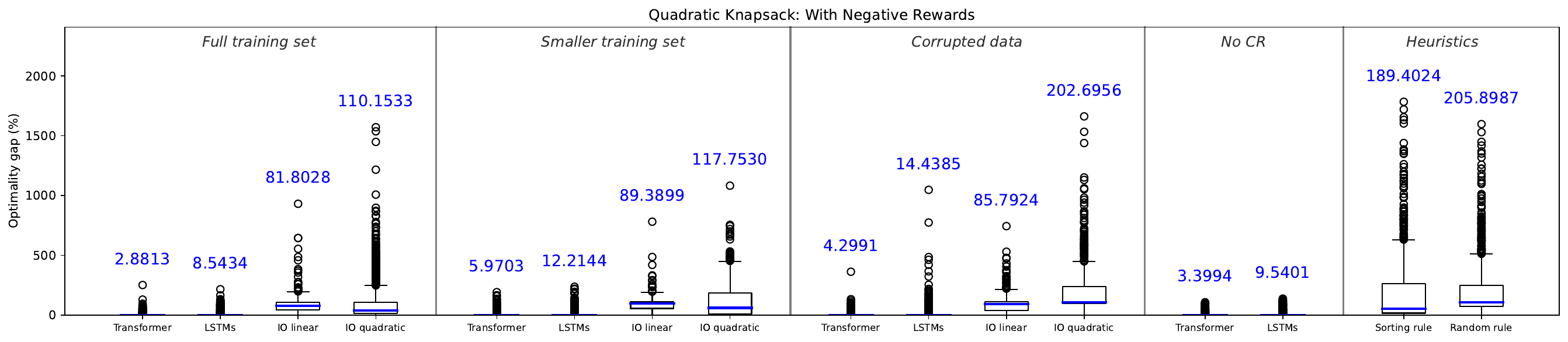}

  \caption{Optimality gaps (\%) for the quadratic instances of the knapsack problem.}
  \label{fig:gaps-quadratic-knapsack}

  \vspace{0.25em}
  {\footnotesize\emph{Note.} Means are displayed on top of the boxes.}
\end{figure}

IO methods perform poorly on quadratic instances. For the linear IO method, this is due to the misspecification of the functional form. The quadratic IO method, while correctly specified, suffers from the complexity of the underlying optimization problem and makes limited progress within the time limit. As a result, the misspecified linear IO method can even outperform the quadratic IO method in this setting.

The {\em Sorting rule} baseline performs well on linear instances but deteriorates sharply on quadratic ones, since it captures only simple ordering patterns and not the pairwise interactions present in these instances. 
This shows that the strong performance of transformers and LSTMs is not an artifact of the sorting rule used in the instance encoding, but reflects the ability of the transformer and LSTM architectures to capture the underlying structure of the problem.

\smallskip

\noindent \textbf{Heuristic Knapsack.~} Table~\ref{tab:results-heuristics} presents the percentage of solutions that satisfy the heuristic rule. 
Transformers deliver the strongest results across all heuristics, reflecting their capacity to represent element interactions and instance-specific latent structures, a consequence of the self-attention mechanism. 
LSTMs rank second with much weaker performance, while the IO approach performs poorly on these instance-dependent heuristics.

\begin{table}[tb]
\centering
\small
\caption{Percentage (\%) of solutions satisfying the heuristic rule.\label{tab:results-heuristics}}
\begin{tabular}{@{}lrrr@{}}
\hline
\textbf{Heuristic/Method} & \textbf{Transformers} & \textbf{LSTMs} & \textbf{Linear IO} \\
\hline
Alternating (choose 1, skip 1)      & 92.93   & 10.05   & 2.10   \\
Alternating (choose 2, skip 1)      & 93.85  & 16.29  & 3.05   \\
Clustering by group                & 91.20  & 13.74  & 3.61  \\
\hline
\end{tabular}
\end{table}

Transformers also offer a clear computational advantage: once trained, they produce solutions in a fraction of a second, with inference time effectively constant across all instance types. LSTMs behave similarly. In contrast, IO must solve a full optimization problem for every observation, so its runtime grows quickly with problem complexity and the assumed functional form; for quadratic knapsack, it often hits the time limit. Full runtime statistics results are reported in Appendix \ref{sec:knap}.

\subsection{Bipartite Matching Problem with an Unknown Objective Function}

Next, we consider the bipartite matching problem.  The optimal solutions are obtained using Gurobi with a 7-hour time limit per instance.  We study instances drawn from a universe of 100 edges from a complete graph; see Appendix~\ref{sec:matching} for more details on instance generation and parameter settings. Given space constraints, we focus here on the most challenging quadratic variant. 
Results for the linear case follow patterns similar to the linear knapsack setting and are reported in Appendix~\ref{sec:matching}.

\medskip

\noindent\textbf{Quadratic Bipartite Matching.~} Figure~\ref{fig:gaps-matching-quadratic} presents the optimality gaps. The results closely mirror those for the quadratic knapsack problem. As in that setting, the quadratic structure introduces pairwise interactions that make IO iterations computationally expensive, yielding little progress within the time limit. This also explains the counterintuitive effect that the quadratic IO method can perform best when trained with less data.
Overall, transformers achieve the smallest optimality gaps across all experiments, followed by LSTMs. While transformer performance is more sensitive to reduced training data than in the knapsack setting, the transformer model still clearly outperform LSTMs and all IO variants.

\begin{figure}[tb]
  \centering
  \includegraphics[width=\textwidth]{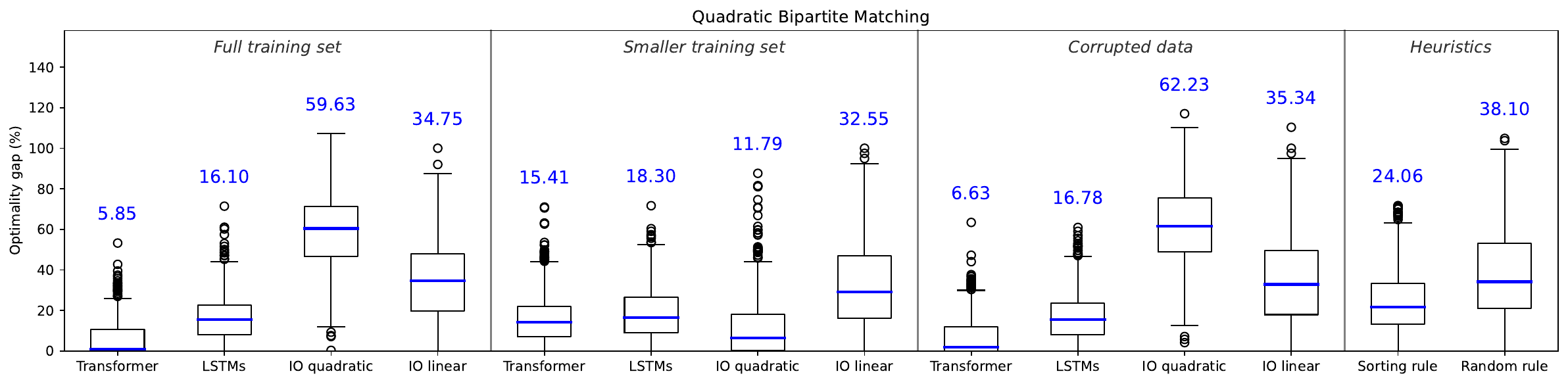}

  \caption{Optimality gaps (\%) for the quadratic bipartite matching problem.}
  \label{fig:gaps-matching-quadratic}

  \vspace{0.25em}
  {\footnotesize\emph{Note.} Means are displayed on top of the boxes.}
\end{figure}

\subsection{Single-Machine Scheduling Problem with Release Times and Unknown Precedence Constraints}

We next report results for the single-machine scheduling problem.
For this problem, optimal solutions are obtained using CODD \citep{michel2024codd}, a decision diagram-based solver for combinatorial optimization.
We study instances of lengths 10 and 20, setting time limits of 3.5 and 12 hours, respectively, for training the ML models in each experiment. Additionally, we allocate a maximum of 2 hours per instance for solution generation with the CODD solver.

Each instance contains a subset of 10 or 20 jobs drawn from a universe of~100. We consider three different topologies of group-based precedence constraints encoded as directed acyclic graphs. Parameter definitions and data generation procedures details are provided in Appendix~\ref{sec:scheduling}.

For the IO approach, we consider three objective functions for the forward problem (see Appendix~\ref{sec:appendix-algo-bench-io} for more details): (i) \textit{IO linear}: a linear objective in job completion times, corresponding to minimizing the average completion time; (ii) \textit{IO jobs}: a linear objective augmented with pairwise job–job interaction terms; and (iii) \textit{IO groups}: a linear objective augmented with group-level interaction terms. The latter two formulations are designed to encode precedence structure. While both are, in principle, expressive enough to represent the underlying implicit precedence constraints when correctly specified, the job-level pairwise interaction model requires estimating a substantially larger number of parameters than the group-level model. An advantage of these two IO formulations is that they explicitly bias the model toward precedence-based structure, whereas the ML-based approaches are not provided with such prior information and instead rely on learning precedence relations from data.

\medskip 
\noindent \textbf{Results.~} For brevity, we report results for a representative set of precedence constraints: Topology A. 
Full results for Topologies B and C are provided in Appendix~\ref{sec:scheduling}. Table~\ref{tab:precedence-results-topA20} summarizes precedence satisfaction percentages across all methods and experimental settings. ML-based methods substantially outperform the IO methods, even though they are trained without any information about the precedence constraints. The transformer model generally achieves the highest satisfaction rates among all approaches. 
In contrast, IO methods infer precedence structure indirectly through the assumed objective variants, which is insufficient: they satisfy the unknown constraints for only a small percentage of instances.
Figure~\ref{fig:gaps-boxplot-20-A} presents the optimality gaps for this topology. We observe that even when assuming the correct functional form, IO exhibits a high mean optimality gap on the few instances that satisfy the precedence constraints compared to other methods. In contrast, the transformer model achieves consistently low optimality gaps while simultaneously maintaining high precedence satisfaction. LSTMs also substantially outperform IO, but remain consistently worse than transformers in both feasibility and solution quality.

\begin{table}[tb]
\centering
\small
\caption{Percentage (\%) of solutions that completely satisfy precedence constraints for the single-machine scheduling problem.\label{tab:precedence-results-topA20}}
\resizebox{\textwidth}{!}{%
\begin{tabular}{@{}lclrrrrrrr@{}}
\hline
\textbf{\begin{tabular}[c]{@{}l@{}}Precedence\\constraints\end{tabular}} & 
\textbf{\begin{tabular}[c]{@{}c@{}}Instance\\length\end{tabular}} & 
\textbf{Experiment} & 
\textbf{Transformers} & 
\textbf{LSTMs} & 
\textbf{\begin{tabular}[c]{@{}c@{}}IO:\\Linear\end{tabular}} & 
\textbf{\begin{tabular}[c]{@{}c@{}}IO:\\Jobs\end{tabular}} & 
\textbf{\begin{tabular}[c]{@{}c@{}}IO:\\Groups\end{tabular}} & 
\textbf{Random} \\
\hline
\multirow{2}{*}{Topology A}  
    & 20 & Full training set       & 100.00 & 99.20 & 0.00 & \begin{tabular}[c]{@{}r@{}}0.00\\(5.00\%)\end{tabular}  & 6.90  & \multirow{2}{*}{0.00} \\
    & 20 & Smaller training set    & 99.30  & 90.60 & 0.00 & 0.00 & 1.80  &                         \\
\hline
\end{tabular}
}
\caption*{\footnotesize The percentage of solved instances is shown in parentheses when it is less than 100\%. For those instances, the percentage reported is among the solved solutions.}
\end{table}

\begin{figure}[t]
  \centering
  \includegraphics[width=0.9\textwidth]{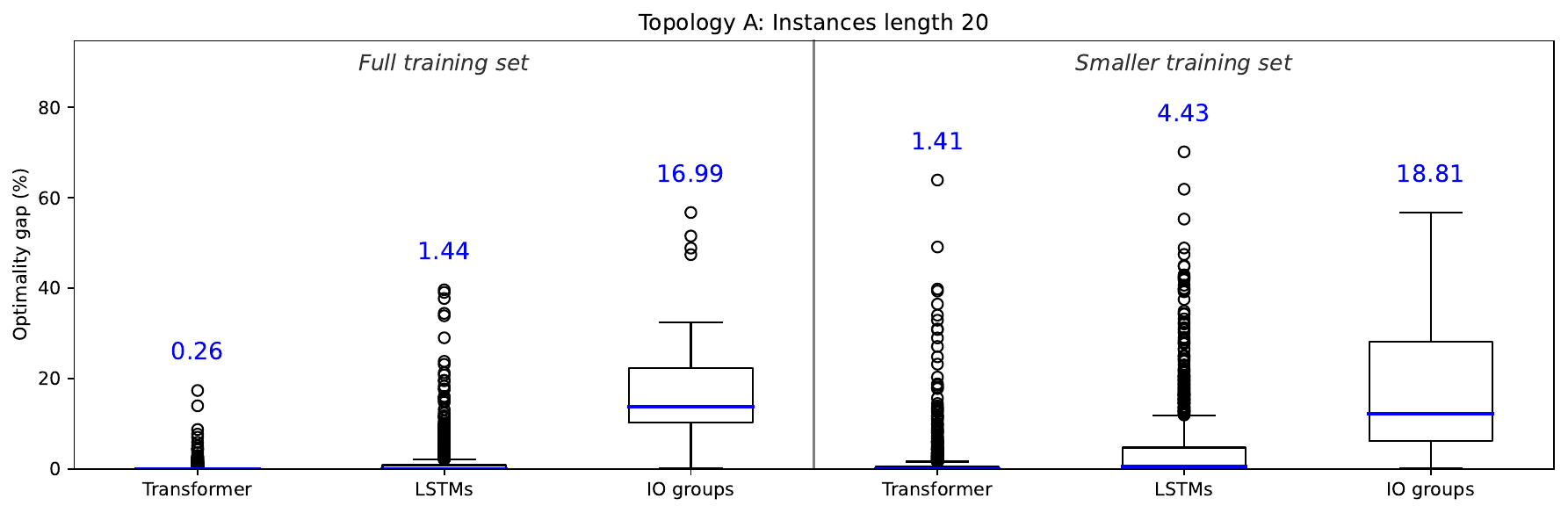}

  \caption{Optimality gaps (\%) for instances of length 20 for the single-machine scheduling problem.}
  \label{fig:gaps-boxplot-20-A}
  \vspace{0.25em}
  {\footnotesize\emph{Note.} These results only consider solutions that satisfy precedence constraints.}
\end{figure}

\subsection{Discussion}

The experiments across the three application problems highlight consistent findings in how transformer models compare with IO and LSTMs. 
These problems instantiate the latent-structure types introduced in Section~\ref{sec:problem}---ranging from observable linear structure, to hidden pairwise interactions, to hidden feasibility constraints---allowing us to assess performance across increasingly complex unknown components.
Here, we synthesize these results, discuss common patterns, and draw broader implications regarding scalability, robustness, and the role of latent problem structures.

We observe that IO performs well when the underlying objective function is simple, e.g. linear, and the model is correctly specified and trained with clean data. However, its effectiveness deteriorates quickly, both in terms of solution quality and generation speed, if the assumed model form is misspecified or the data is noisy, demonstrating that it is not robust. As discussed in Section~\ref{sec:algorithmic-benchmarks}, this is because the cutting-plane algorithm that it uses assumes uncorrupted data and, at each iteration, must solve one forward problem per observation. When the forward problems are complex, this process becomes slow and may fail to make sufficient progress before reaching the time limit.

LSTM-based models, in turn, are generally more robust than IO, since they do not rely on strong assumptions about the functional form of the objective and are less sensitive to data imperfections. However, they are consistently outperformed by transformers, particularly in settings where the objective function or implicit constraints depend on relationships between elements. In these cases, the attention mechanism described in Section~\ref{sec:method} enables transformers to capture structural dependencies that LSTMs cannot. That said, when the underlying problem structure is simple, LSTMs can provide a competitive and more computationally efficient alternative, since their training and inference costs are substantially lower than those of transformers. The key challenge, however, is that in practice the complexity of the latent structure is rarely known in advance. This makes transformers the most reliable choice: across all applications and experimental configurations, they consistently delivered strong performance, establishing them as the safest option when tackling optimization problems with unknown components.

The performance of ML-based methods also varies depending on the type of latent structure. The cases where the latent structure depends only on observable parameters are the easiest to learn because the representation bias from feature encoding strongly aids learning. It is more challenging when the latent structures depend on unobservable parameters. Since the mapping from features to objectives or constraints is hidden, models must rely more on architectural bias (e.g., attention) and on the generalization patterns learned from past decisions. Finally, the instance-specific type is the most challenging. In this case, inductive bias must be flexible enough to adapt across instances rather than discover a stable rule. Naturally, performance is shaped not only by the latent structure but also by the inherent complexity of the underlying optimization problem.

\section{Conclusions}
\label{sec:conclusions}

We introduced a structured prediction framework for discrete optimization problems with unknown components, combining transformer models with DFA-based constraint reasoning to generate feasible solutions directly from instance data. By learning the mapping from inputs to optimal decisions, our approach circumvents the modeling assumptions and computational difficulties inherent to IO, particularly under misspecification or hidden feasibility constraints. Across three classical applications---the knapsack problem, bipartite matching, and single-machine scheduling---transformers consistently delivered high-quality solutions and substantial inference-time gains relative to IO and LSTM-based baselines, demonstrating the effectiveness of attention-based architectures in capturing latent problem structure. While our approach relies on monotone constraint systems and sufficient historical solution data, these assumptions arise naturally in many practical settings. Overall, the results highlight structured prediction as a viable alternative to IO when objectives or constraints are unknown or difficult to specify.

\bibliography{sample}

\newpage
\section*{Appendix}
    \renewcommand{\thesubsection}{\Alph{subsection}.}
    \renewcommand{\thefigure}{\thesubsection\arabic{figure}}
    \renewcommand{\thesubsubsection}{\thesubsection\arabic{subsubsection}}
    \renewcommand{\thetable}{\thesubsection\arabic{table}}
    \setcounter{figure}{0}
    \setcounter{table}{0}

The following Electronic Companion provides additional methodological details, complete experimental configurations, and extended results that complement the main text.

\subsection{Instance Generation}
\label{sec:appendix-instances}

\subsubsection{Knapsack Problem}
\label{sec:appendix-instances-knap}

We consider a universe of $n$ elements. We randomly sample each $w_{j}$ from $[1,K]$. Each instance is defined by a subset $U$ of elements and a maximum capacity $B$. We first draw the subset $U$ and then define the maximum capacity $B = \frac{p}{P + 1}\sum_{j \in S}
w_{j}$, with $P = 100$ and $p \in \{1,...,100\}$. We take a sample of five values $p$ for each element subset, obtaining five different instances that differ in their maximum capacities.

\subsubsection{Bipartite Matching Problem}
\label{sec:appendix-instances-match}
We construct a universe of $n_L$ nodes on the left-hand side and $n_R$ nodes on the right-hand side. From these, we generate a complete bipartite graph with node set $N = V_{1} \cup V_{2}$, where $|V_{1}| = n_{L}$ and $|V_{2}| = n_{R}$, and edge set $E = \{(j,k): j \in V_{1}, k \in V_{2}\}$. For each instance, we sample a subset of edges $S_{E} \subseteq E$, with their corresponding labels based on group combinations.

In this application, inspired by practical allocation problems such as team formation and diversity-aware task assignment, we define rewards based on group-based interactions (see Table~\ref{tab:group_rewards-matching}). For linear rewards, each edge receives a value that depends on the absolute difference between the group labels of its endpoints, which reflects the affinity between, for example, a worker and a task. A reward is sampled uniformly from a group-specific range and then applied consistently across all instances.

For the quadratic case, we define pairwise interactions among left-side nodes (e.g., workers) assigned to the same right-side node (e.g., tasks). In particular, we penalize assignments in which all nodes belong to the same group and reward those where exactly one node differs in group membership. This design encourages within-task diversity while discouraging homogeneous (monocultural) assignments.

\subsubsection{Single-machine Scheduling Problem}
\label{sec:appendix-instances-scheduling}

We construct a universe of $n$ jobs. For each job $j \in U$ in group $g_{j} \in G$, its processing time $p_{j}$ is sampled from a truncated normal distribution, denoted as $p_{j} \sim TN(\mu_{g_{j}}, \sigma_{g_{j}}; [1,K])$, where $(\mu_{g_{j}}, \sigma_{g_{j}})$ are group-specific parameters and $K$ is a global upper bound shared by all groups. For each instance, a subset $S \subseteq U$ of jobs is selected, and for each job $j \in S$, a release time $r_{j} \in \mathbb{Z}^{+}$ is sampled following the OAS benchmark procedure \citep{de2021single}, which depends on a tardiness factor $\tau$. Specifically, release times are drawn uniformly at random from the interval $[0, \tau \cdot \sum_{j \in S}p_{j}]$. Without loss of generality, both processing and release times are rounded to the nearest integer. Additionally, we introduce group-based precedence constraints in the form of directed acyclic graphs, which encode underlying structural preferences as constraints. These constraints are fixed across all instances and are determined by the group assignments.

\subsection{Algorithmic Benchmarks}
\label{sec:appendix-algo-bench}

\subsubsection{Inverse Optimization Approach}
\label{sec:appendix-algo-bench-io}



Here, we specify the problem-specific
forward models, parametric objective families, initialization schemes, and objective variants for the three applications.

\smallskip

\noindent \textbf{Knapsack Problem.~} The unknown parameters correspond to $c = (c^{1}, c^{2})$. We assume two different forward problems depending on the experiment, given by $\textbf{KP-L}(c^{1}, u)$ or $\textbf{KP-Q}(c, u)$. For the initial guess $c^{0}$, we use the same empirical frequencies employed to sort the input and target data. In the quadratic case, these frequencies are computed for each pair of elements rather than for individual elements.

\medskip

\noindent \textbf{Bipartite Matching Problem.~} The unknown parameters correspond to $c = (c^{1}, c^{2})$. We assume two different forward problems, given by $\textbf{BM-L}(c^{1}, u)$ for a linear objective or $\textbf{BM-Q}(c, u)$ for a quadratic one. As in the knapsack setting, we construct the initial guesses $c^{0}$ from empirical frequencies of appearance in the optimal solutions, at the edge level for the linear case and at the pairwise level for the quadratic case.
\medskip

\noindent \textbf{Scheduling Problem.~} In this application, the unknown components are the precedence constraints. In the forward problem, we keep the constraints of \textbf{SMS}$(\mathcal{Q}, u)$. For the objective, we assume that it corresponds to minimizing the average completion time, but we also explore alternative functional forms to capture the presence of implicit precedence constraints through $f_{\text{inv}}$. This gives the IO approach an advantage compared to ML-based methods, which have to learn that there are precedence constraints.

We consider three objective functions for $f_{\text{inv}}$, presented in Table~\ref{tab:io-scheduling}, where $c = (c^{1}, c^{2}, c^{3})$ are the unknown parameters: (i) a linear objective in job completion times; (ii) a linear objective augmented with pairwise job–job interaction terms; and (iii) a linear objective augmented with group-level interaction terms. The interaction terms in the second and third objective functions encode precedence structure. When the precedence coefficients $c^{2}$ and $c^{3}$ are correctly specified and sufficiently large relative to the linear weights $c^{1}$, the second and third objective recover the true underlying objective that captures precedence constraints.

To model the third objective function, additional variables and constraints are required. We introduce a binary matrix variable $m_{gh} \in \{0.1\}^{|G|\cdot |G|}$ where $m_{gh} = 1$ indicates that all jobs in group $g \in G$ are scheduled before all jobs in group $h \in G$. We add the following constraints to \textbf{SMS}$(\mathcal{Q}, u)$. Let $J_{g}$ be the set of jobs in group $g$:
\begin{equation}
\label{eq:new-constraints}
\begin{aligned}
m_{gh} &\le z_{jk} \quad \forall\, g,h \in G,\; g \neq h,\; j \in J_g,\; k \in J_h&&  \\
1- m_{gh} &\le \sum_{j \in J_g}\sum_{k \in J_h} \bigl(1-z_{jk}\bigr) \quad \forall\, g,h \in G,\; g \neq h&&
\end{aligned}
\end{equation}

\begin{table}[tbp]
\centering
\small
\caption{Summary of forward problem objectives and their corresponding IO decision variables.\label{tab:io-scheduling}}
\begin{tabular}{@{}lcl@{}}
\toprule
\multicolumn{1}{c}{\textbf{Reward type}} & \textbf{Name} &
\multicolumn{1}{c}{\textbf{Function and variables}} \\
\midrule
\multirow{2}{*}{Linear objective ($c^{1}$)} & \multirow{2}{*}{IO linear} &
$\displaystyle \sum_{j \in U}c^{1}_{j}\,w_{j}$ \\
\cmidrule(l){3-3}
& & $c^{1} \in \mathbb{R}^{|U|}$ \\
\midrule
\multirow{2}{*}{Jobs pairwise rewards ($c^{1}$, $c^{2}$)} & \multirow{2}{*}{IO jobs} &
$\displaystyle \sum_{j \in U}c^{1}_{j}\,w_{j}
\;+\; \sum_{\substack{j,k \in U\\ j \neq k}}c^{2}_{jk}\,z_{jk}$ \\
\cmidrule(l){3-3}
& & $c^{1} \in \mathbb{R}^{|U|}_{\ge 0},\;\; c^{2} \in \mathbb{R}^{|U|\cdot |U|}$ \\
\midrule
\multirow{2}{*}{Group pairwise rewards ($c^{1}$, $c^{3}$)} & \multirow{2}{*}{IO groups} &
$\displaystyle \sum_{j \in U}c^{1}_{j}\,w_{j}
\;+\; \sum_{\substack{g,h \in G\\ g \neq h}}c^{3}_{gh}\,m_{gh}$ \\
\cmidrule(l){3-3}
& & $c^{1} \in \mathbb{R}^{|U|}_{\ge 0},\;\; c^{3} \in \mathbb{R}^{|G|\cdot |G|}$ \\
\bottomrule
\end{tabular}
\caption*{\footnotesize Each row corresponds to a different hypothesis about the unknown objective function used by the decision maker.}
\end{table}

In this case, the algorithm requires an initial guess $c^{0} = (c^{1,0}, c^{2,0}, c^{3,0})$ for the linear and quadratic cost parameters $c^{1}$, $c^{2}$, and $c^{3}$. To initialize the linear costs $c^{1,0}$, we compute each job’s average completion time across all training instances and set its initial cost as the inverse of this value, giving higher weight to jobs that consistently finish earlier. The quadratic job costs $c^{2,0}$ are estimated from the empirical frequency with which one job precedes another when both appear in the same instance, capturing pairwise precedence preferences. Similarly, the group-level quadratic costs $c^{3,0}$ are initialized using the frequency with which one group precedes another across training instances.

\subsubsection{Additional Baseline Approaches}

In addition to IO and LSTM-based methods, we apply different baseline approaches for comparison, for each of the applications.

\medskip

\noindent \textbf{Knapsack Problem.~} We compare against simple constructive heuristics that select items based on randomized or rule-based orderings. In the \textit{Random} baseline, the set of items is randomly permuted to form a priority list, and items are added to the knapsack in that order until the capacity constraint is reached. This baseline is independent of the objective function. The \textit{Omniscient greedy} heuristic orders items in decreasing reward-to-weight ratio and adds them sequentially until the capacity is reached; this applies only when the value function is linear. Finally, the \textit{Sorting rule} heuristic orders items according to the predefined sorting rule in Section~\ref{sec:application-problems}, and adds them in that order, provided the capacity constraint is not violated.

\medskip

\noindent \textbf{Bipartite Matching Problem.~} We implement heuristics that prioritize edges using randomized or deterministic sorting rules. In the \textit{Random} baseline, the set of edges is randomly permuted to create a priority list, and the edges are greedily selected subject to feasibility constraints. The \textit{Sorting rule} heuristic orders edges according to the sorting rule defined in Section~\ref{sec:application-problems} and adds them to the matching in that order, provided the feasibility is maintained.

\medskip

\noindent \textbf{Scheduling Problem.~} We use a simple randomized baseline to generate job sequences. The \textit{Random} baseline randomly permutes the set of jobs to generate a schedule.

\subsection{Performance Metrics and Experimental Configurations}
\label{sec:exp-conf}

\noindent \textbf{Knapsack Problem with an Unknown Reward Function.~} We conduct two sets of experiments. The first explores the linear and quadratic knapsack problems, while the second explores the heuristic knapsack problem. For the first set of experiments, 
we study the deterministic value functions presented in Table~\ref{tab:functional-forms}. For each of these functions, we conduct all four experimental variations. To corrupt the data, we use the \textit{Random} baseline heuristic. 

In addition, we evaluate the IO method on quadratic problems under the incorrect assumption of a linear functional form, providing insights into its performance when it inaccurately models the underlying functional form of the unknown component.

The second set of experiments studies the three heuristic rules introduced in Section~\ref{sec:application-problems}: \textit{Alternating (choose 1, skip 1)}, \textit{Alternating (choose 2, skip 1)}, and \textit{Clustering by group}.
We only conduct the first experimental evaluation for these experiments.

For the first set of experiments on linear and quadratic knapsack problems, we consider two metrics: optimality gap and percentage of instances solved to optimality. For the experiments on the heuristic knapsack, the numerical notion of optimality gap is no longer suitable, since our solutions are derived from rule-based heuristics.  In this case, we measure the performance by the percentage of solutions in the validation set that satisfy the corresponding heuristic rule. 

\medskip

\noindent \textbf{Bipartite Matching Problem with an Unknown Objective Function.~}
We adopt the same experimental design as in the knapsack setting. For each value function, we apply the first three experimental variations introduced above: training on the full dataset, training on 10\% of the data, and training on data corrupted by 5\% of suboptimal solutions generated by the \textit{Random} baseline.

As in the knapsack experiments, we also evaluate the IO method under model misspecification by assuming a linear functional form for instances generated following the quadratic bipartite matching problem.

We evaluate performance using the same metrics as in the first set of knapsack experiments: the optimality gap and the percentage of instances solved to optimality.

\medskip

\noindent \textbf{Single-machine Scheduling Problem with Release Times and Unknown Precedence Constraints.~} We conduct the first two experimental variations defined above: training on the full dataset and training on only 10\% of the data. In addition, and consistent with the other applications, we evaluate the IO algorithm under model misspecification by applying the first objective presented in Table~\ref{tab:io-scheduling}, since both the second and third one can capture the underlying precedence constraints.

To evaluate whether the methods are successfully capturing the latent precedence constraints, we report the percentage of solutions that fully satisfy the precedence constraints. Among these feasible solutions, we further assess quality by comparing their average completion times with those of the optimal schedules and reporting the corresponding optimality gap. Additionally, we compute the edit distances with respect to optimal solutions.

\subsection{Implementation Details}
\label{sec:implementation}

Across all applications, we incorporate a GELU activation at the input embedding stage to enhance the expressiveness of the learned representations. We also apply dropout with a rate of 0.2 to mitigate overfitting and improve generalization.

Regarding the learning dynamics, we use the SOAP optimizer with a batch size of 1024, reduced to 512 when memory constraints arose. We impose a fixed time limit for each experiment, and stopped training ML models either upon signs of overfitting or once the time limit was reached. We use 1,000 observations for evaluation.

\subsection{Parameter Specification and Extended Results}

\subsubsection{Knapsack Problem}
\label{sec:knap}

\noindent\textbf{Parameter Definition.~}
We consider a universe of $100$ elements and subsets of varying sizes,  $|S| = 10, 20,\dots, 90$. The datasets contain approximately 90,000 observations. We set $K = 10,000$. We define the hyperparameter values of the proposed and competing methods. Table \ref{tab:parameters-knapsack} outlines the configurations used in the experiments. For the ML-based models, we select the hyperparameter settings that achieve the best performance in terms of loss values during validation. For the IO cutting-plane algorithm, we adopt the parameter values of~\cite{bodur2022inverse} and use $\lambda = 5$. 
We define the constants $q_{1}$ to $q_{3}$ based on the ranges of the reward value functions. In particular, we find that $q_{1} = 0.0015$, $q_{2} = 0.0003$, and $q_{3} = 0.0009$ yield the most interesting results.

\begin{table}[tb]
\centering
\small
\caption{Parameter configuration for the machine learning models used in the knapsack problem.\label{tab:parameters-knapsack}}
\begin{tabular}{lllrr}
\hline
\textbf{Method} & \textbf{Instance type} & \textbf{Parameter} & \textbf{Value} & \textbf{Number of Parameters} \\
\hline
\multirow{5}{*}{Transformers} 
    & \multirow{5}{*}{\begin{tabular}[c]{@{}l@{}}Linear inverse\\ and proportionality\end{tabular}} 
    & Dimension       & 96   & \multirow{5}{*}{3,642,805} \\
    &                                & Number of heads & 4    &                           \\
    &                                & Encoder layers  & 4    &                           \\
    &                                & Decoder layers  & 4    &                           \\
    &                                & Learning rate   & 1e-4 &                           \\
\hline
\multirow{5}{*}{Transformers} 
    & \multirow{5}{*}{\begin{tabular}[c]{@{}l@{}}Linear logarithmic\\ and quadratic\end{tabular}} 
    & Dimension       & 128  & \multirow{5}{*}{5,024,124} \\
    &                                & Number of heads & 4    &                           \\
    &                                & Encoder layers  & 4    &                           \\
    &                                & Decoder layers  & 4    &                           \\
    &                                & Learning rate   & 1e-4 &                           \\
\hline
\multirow{3}{*}{LSTMs} 
    & \multirow{3}{*}{Linear and quadratic} 
    & Embedding dimension & 128 & \multirow{3}{*}{315,077} \\
    &                                & Hidden dimension    & 128 &                          \\
    &                                & Learning rate       & 5e-4&                          \\
\hline
\end{tabular}
\end{table}

\medskip

\noindent\textbf{Computational Speed.~} To better understand the trade-off between accuracy and efficiency, we report both training and solution generation times of each method on the different instances. Table~\ref{tab:training-speed-linear-knap}, Table~\ref{tab:training-times-quad-knap} and Table~\ref{tab:training-times-heuristic-knap} report the training and solution generation times for linear, quadratic and heuristic instances, respectively.

\begin{table}
\centering
\small
\caption{Training (hours) and mean inference (seconds) times of competing methods on linear knapsack instances.\label{tab:training-speed-linear-knap}}
\resizebox{\textwidth}{!}{%
\begin{tabular}{@{}llrrrrrr@{}}
\hline
\multirow{2}{*}{\textbf{Instance}} & \multirow{2}{*}{\textbf{Experiment}}
& \multicolumn{2}{c}{\textbf{Transformers}}
& \multicolumn{2}{c}{\textbf{LSTMs}}
& \multicolumn{2}{c}{\textbf{Linear IO}} \\
& & \textbf{Training} & \textbf{Inference} & \textbf{Training} & \textbf{Inference} & \textbf{Training} & \textbf{Inference} \\
\hline
\multirow{4}{*}{\begin{tabular}[c]{@{}l@{}}Linear inverse\\ proportionality\end{tabular}}
  & Full training set       & 7.56 & 0.040 & 3.20 & 0.005 & 3.10 & 0.008 \\
  & Smaller training set    & 5.42 & 0.040 & 1.42 & 0.005 & 0.38 & 0.008 \\
  & Corrupted data          & 11.33 & 0.040 & 1.60 & 0.005 & 12.00 & 0.010 \\
  & No CR during training   & 9.44 & 0.040 & 1.40 & 0.005 & N/A & N/A \\
\hline
\multirow{4}{*}{\begin{tabular}[c]{@{}l@{}}Linear\\ algorithmic\end{tabular}}
  & Full training set       & 9.72 & 0.040 & 5.00 & 0.005 & 4.15 & 0.009 \\
  & Smaller training set    & 9.11 & 0.040 & 1.20 & 0.005 & 0.68 & 0.009 \\
  & Corrupted data          & 9.72 & 0.040 & 2.80 & 0.005 & 12.00 & 0.012 \\
  & No CR during training   & 9.72 & 0.040 & 1.80 & 0.005 & N/A & N/A \\
\hline
\end{tabular}}
\caption*{\footnotesize CR denotes constraint reasoning and N/A indicates not applicable.}
\end{table}

\begin{table}[tb]
\centering
\small
\caption{Training (hours) and mean inference (seconds) times of competing methods on quadratic knapsack instances.\label{tab:training-times-quad-knap}}
\resizebox{\textwidth}{!}{%
\begin{tabular}{@{}lrrrrrrrr@{}}
\hline
\multirow{2}{*}{\textbf{Experiment}}
& \multicolumn{2}{c}{\textbf{Transformers}}
& \multicolumn{2}{c}{\textbf{LSTMs}}
& \multicolumn{2}{c}{\textbf{Linear IO}}
& \multicolumn{2}{c}{\textbf{Quadratic IO}} \\
& \textbf{Training} & \textbf{Inference} & \textbf{Training} & \textbf{Inference} & \textbf{Training} & \textbf{Inference} & \textbf{Training} & \textbf{Inference} \\
\hline
Full training set        & 11.66 & 0.039 & 4.83 & 0.005 & 12.00 & 0.009 & 12.00 & 4.210 \\
Smaller training set     & 6.66  & 0.039 & 2.67 & 0.005 & 12.00 & 0.009 & 12.00 & 6.935 \\
Corrupted data           & 11.55 & 0.039 & 1.67 & 0.005 & 12.00 & 0.009 & 12.00 & 8.806 \\
No CR during training    & 12.00 & 0.039 & 4.83 & 0.005 & N/A & N/A & N/A & N/A \\
\hline
\end{tabular}
}
\caption*{\footnotesize CR denotes constraint reasoning and N/A indicates not applicable.}
\end{table}

\begin{table}
\small
\centering
\caption{Training (hours) and mean inference (seconds) times of competing methods on the heuristic knapsack instances.\label{tab:training-times-heuristic-knap}}
\begin{tabular}{@{}lrrrrrr@{}}
\hline
\multirow{2}{*}{\textbf{Heuristic/Method}} 
& \multicolumn{2}{c}{\textbf{Transformers}} 
& \multicolumn{2}{c}{\textbf{LSTMs}} 
& \multicolumn{2}{c}{\textbf{Linear IO}} \\
& \textbf{Training} & \textbf{Inference} & \textbf{Training} & \textbf{Inference} & \textbf{Training} & \textbf{Inference} \\
\hline
Alternating (choose 1, skip 1) & 4.03 & 0.020 & 1.26 & 0.002 & 12.00 & 0.004 \\
Alternating (choose 2, skip 1) & 7.70 & 0.020 & 2.10 & 0.002 & 12.00 & 0.004 \\
Clustering by group            & 2.46 & 0.020 & 0.05 & 0.002 & 12.00 & 0.004 \\
\hline
\end{tabular}
\end{table}

Solution generation speed for machine learning-based methods is uniform regardless of the type of instance. Using the IO method, the speed strongly depends on the underlying optimization problem and the assumed functional form. In fact, the IO method times out on the smaller training set and under the corrupted data configurations for the quadratic instances, respectively.

\smallskip

\noindent \textbf{Quadratic Knapsack.~} Table~\ref{tab:quadratic-knap-0gaps} reports the percentage of optimally solved instances for the set of quadratic instances.

\begin{table}[htb]
\centering
\small
\caption{Percentage (\%) of solutions solved to optimality for the quadratic instances of the knapsack problem.\label{tab:quadratic-knap-0gaps}}
\resizebox{\textwidth}{!}{%
\begin{tabular}{lrrrrrr}
\hline
\textbf{Experiment} & \textbf{Transformers} & \textbf{LSTMs} & \textbf{Linear IO} & \textbf{Quadratic IO} & \textbf{Sorting rule} & \textbf{Random rule} \\
\hline
Full training set     & 86.21 & 70.83 & 7.49 & 4.10                  & \multirow{4}{*}{0.70} & \multirow{4}{*}{0.40} \\
Smaller training set  & 79.12 & 66.93 & 6.29 & 3.60 (72.20\%)        &                       &                      \\
Corrupted data        & 85.41 & 69.83 & 6.89 & 0.00 (59.80\%)        &                       &                      \\
No CR during training & 83.32 & 68.33 & N/A  & N/A                   &                       &                      \\
\hline
\end{tabular}
}
\caption*{\footnotesize The percentage of solved instances shown in parentheses when it is less than 100\%. For those instances, the percentage
reported is among the solved solutions.}
\end{table}

\subsubsection{Bipartite Matching Problem}
\label{sec:matching}

\noindent\textbf{Parameter Definition.~} We consider a complete bipartite graph with $n_L = 10$ left nodes and $n_R = 10$ right nodes. To construct instances with varying sparsity, we generate edge subsets by sampling edges with different selection probabilities: 0.2, 0.4, and 0.6. We have 3 groups to which nodes belong. The dataset contains 200,000 observations. We define the hyperparameter values of the proposed and competing method, Table~\ref{tab:ml-config-matching} outlines the configurations used in the experiments. For the ML-based methods, we select the configurations that yielded the best loss values during validation. For the IO cutting-plane algorithm we use the same values as for the knapsack problem application because we observe they worked well. We define the constants for the linear and quadratic rewards as $a_{0} = 8$, $b_{0} = 10$, $a_{1} = 4$, $b_{1} = 6$, $a_{2} = 1$, $b_{2} = 2$, $\gamma = 6$.

\begin{table}[htb]
\centering
\small
\caption{Parameter configuration for the machine learning models used in the bipartite matching problem.\label{tab:ml-config-matching}}
\begin{tabular}{llrr}
\hline
\textbf{Method} & \textbf{Parameter} & \textbf{Value} & \textbf{Number of Parameters} \\
\hline
\multirow{5}{*}{Transformers} 
    & Dimension         & 128   & \multirow{5}{*}{6,234,181} \\
    & Number of heads   & 8     &                           \\
    & Encoder layers    & 6     &                           \\
    & Decoder layers    & 4     &                           \\
    & Learning rate     & 1e-4  &                           \\
\hline
\multirow{3}{*}{LSTMs} 
    & Embedding dimension & 128   & \multirow{3}{*}{315,077} \\
    & Hidden dimension    & 128   &                          \\
    & Learning rate       & 5e-4  &                          \\
\hline
\end{tabular}
\end{table}

\medskip

\noindent\textbf{Computational Speed.~} Table~\ref{tab:times-training-bipartite} reports the training times for linear and quadratic instances of the bipartite matching problem, and the solution generation times. As in the knapsack problem, the solution generation speed for machine learning-based methods is uniform regardless of the type of instance. In contrast, IO inference time is dictated by the underlying forward problem. For quadratic instances, mean solution generation time is more than twice that of linear instances.

\begin{table}
\centering
\small
\caption{Training (hours) and mean inference (seconds) times of competing methods on the bipartite matching problem instances.\label{tab:times-training-bipartite}}
\resizebox{\textwidth}{!}{%
\begin{tabular}{@{}llrrrrrrrr@{}}
\hline
\multirow{2}{*}{\textbf{Instance}} & \multirow{2}{*}{\textbf{Experiment}}
& \multicolumn{2}{c}{\textbf{Transformers}}
& \multicolumn{2}{c}{\textbf{LSTMs}}
& \multicolumn{2}{c}{\textbf{Linear IO}}
& \multicolumn{2}{c}{\textbf{Quadratic IO}} \\
& & \textbf{Training} & \textbf{Inference} & \textbf{Training} & \textbf{Inference} & \textbf{Training} & \textbf{Inference} & \textbf{Training} & \textbf{Inference} \\
\hline
\multirow{3}{*}{Linear}
  & Full training set       & 4.38 & 0.009 & 6.88 & 0.001 & 2.45 & 0.002 & N/A   & N/A   \\
  & Smaller training set    & 0.72 & 0.009 & 3.61 & 0.001 & 0.28 & 0.002 & N/A   & N/A   \\
  & Corrupted data          & 5.93 & 0.009 & 6.95 & 0.001 & 7.00 & 0.003 & N/A   & N/A   \\
\hline
\multirow{3}{*}{Quadratic}
  & Full training set       & 6.67 & 0.009 & 7.00 & 0.001 & 7.00 & 0.002 & 7.00 & 0.008 \\
  & Smaller training set    & 0.66 & 0.009 & 3.61 & 0.001 & 7.00 & 0.002 & 1.33 & 0.005 \\
  & Corrupted data          & 4.81 & 0.009 & 6.68 & 0.001 & 7.00 & 0.003 & 7.00 & 0.007 \\
\hline
\end{tabular}
}
\end{table}

\smallskip

\noindent\textbf{Linear Bipartite Matching.~} Figure~\ref{fig:gaps-matching-linear} reports the optimality gaps for the sets of linear instances. As expected, and consistent with the knapsack experiments, the IO method assuming a linear objective performs exceptionally well, achieving the best results among all methods. However, its performance drops substantially, matching that of the random heuristic, when the data is corrupted. Table \ref{tab:gap0-linear-matching} reports the percentage of optimally solved instances for the linear bipartite matching instances.

\begin{table}
\centering
\small
\caption{Percentage (\%) of solutions solved to optimality for the linear instances of the bipartite matching problem.\label{tab:gap0-linear-matching}}
\begin{tabular}{lrrrrr}
\hline
\textbf{Experiment} & \textbf{Transformers} & \textbf{LSTMs} & \textbf{Linear IO} & \textbf{Sorting rule} & \textbf{Random} \\
\hline
Full training set    & 52.10 & 20.40 & 99.70 & \multirow{3}{*}{14.20} & \multirow{3}{*}{0.40} \\
Smaller training set & 25.80 & 18.90 & 97.90 &                       &                       \\
Corrupted data       & 52.30 & 20.80 & 1.70     &                       &                       \\
\hline
\end{tabular}
\end{table}

Transformers achieve the next best performance after IO, maintaining robustness as their results are not significantly affected by corrupted data. LSTMs follow a similar trend to transformers across the three experiments, but consistently yield inferior results. All ML-based methods experience a non-negligible performance decline when trained on reduced datasets. Notably, while the ML-based methods employ a sorting rule to encode inputs and outputs, simply following this sorting rule results in substantially higher optimality gaps.

\begin{figure}[htb]
  \centering
  \includegraphics[width=\textwidth]{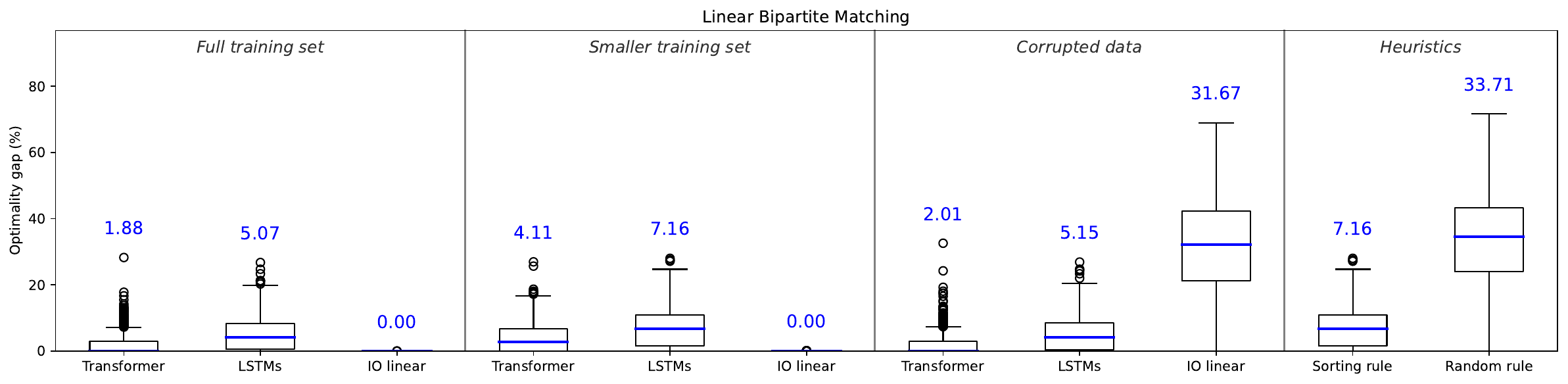}

  \caption{Optimality gaps (\%) for the linear bipartite matching problem.}
  \label{fig:gaps-matching-linear}

  \vspace{0.25em}
  {\footnotesize\emph{Note.} Means are displayed on top of the boxes.}
\end{figure}

\smallskip
\noindent\textbf{Quadratic Bipartite Marching.~} Table \ref{tab:gap0-quadratic-matching} reports the percentage of optimally solved instances for the quadratic bipartite matching instances.

\begin{table}[htb]
\small
\centering
\caption{Percentage (\%) of solutions solved to optimality for the quadratic instances of the bipartite matching problem.\label{tab:gap0-quadratic-matching}}
\begin{tabular}{lrrrrrr}
\hline
\textbf{Experiment} & \textbf{Transformers} & \textbf{LSTMs} & \textbf{Quadratic IO} & \textbf{Linear IO} & \textbf{Sorting rule} & \textbf{Random} \\
\hline
Full training set    & 42.40 & 8.90 & 0.00  & 1.50 & \multirow{3}{*}{4.90} & \multirow{3}{*}{0.20} \\
Smaller training set & 11.60 & 6.80 & 20.10 & 2.40 &                       &                       \\
Corrupted data       & 37.10 & 10.00 & 0.00    & 1.10   &                       &                       \\
\hline
\end{tabular}
\end{table}

\subsubsection{Single-machine Scheduling Problem}
\label{sec:scheduling}

For this application, we solve each instance to optimality using CODD \citep{michel2024codd}, a decision diagram-based solver for combinatorial optimization. We formulate the problem using dynamic programming, enhanced with a state merging mechanism to construct relaxed decision diagrams, and a local bounding function to improve efficiency. We also experiment with IBM ILOG CP Optimizer and Gurobi; however, both solvers prove computationally impractical for the number of required instances, even at moderate problem sizes.

For instances of lengths 10 and 20, we set time limits of 3.5 and 12 hours, respectively, for training in each experiment. Additionally, we allocate a maximum of 2 hours for solution generation. We use the SOAP optimizer to train the machine learning models, with a batch size of 1024 because it works better with larger batch sizes.

\smallskip

\noindent\textbf{Parameter Definition.~} We consider a universe of $100$ jobs and subsets of fixed size $|S|$. We study two different sizes, $|S| = 10, 20$. We have $5$ groups, where the processing times of their jobs are sampled from group-dependent truncated normal distributions, with group means of 1000, 3000, 5000, 7000 and 9000, respectively, standard deviation of 400, and $K = 10^{4}$. Figure~\ref{fig:normal-dist} shows the probability density functions for each group.

\begin{figure}[tb]
  \centering
  \includegraphics[width=0.85\textwidth]{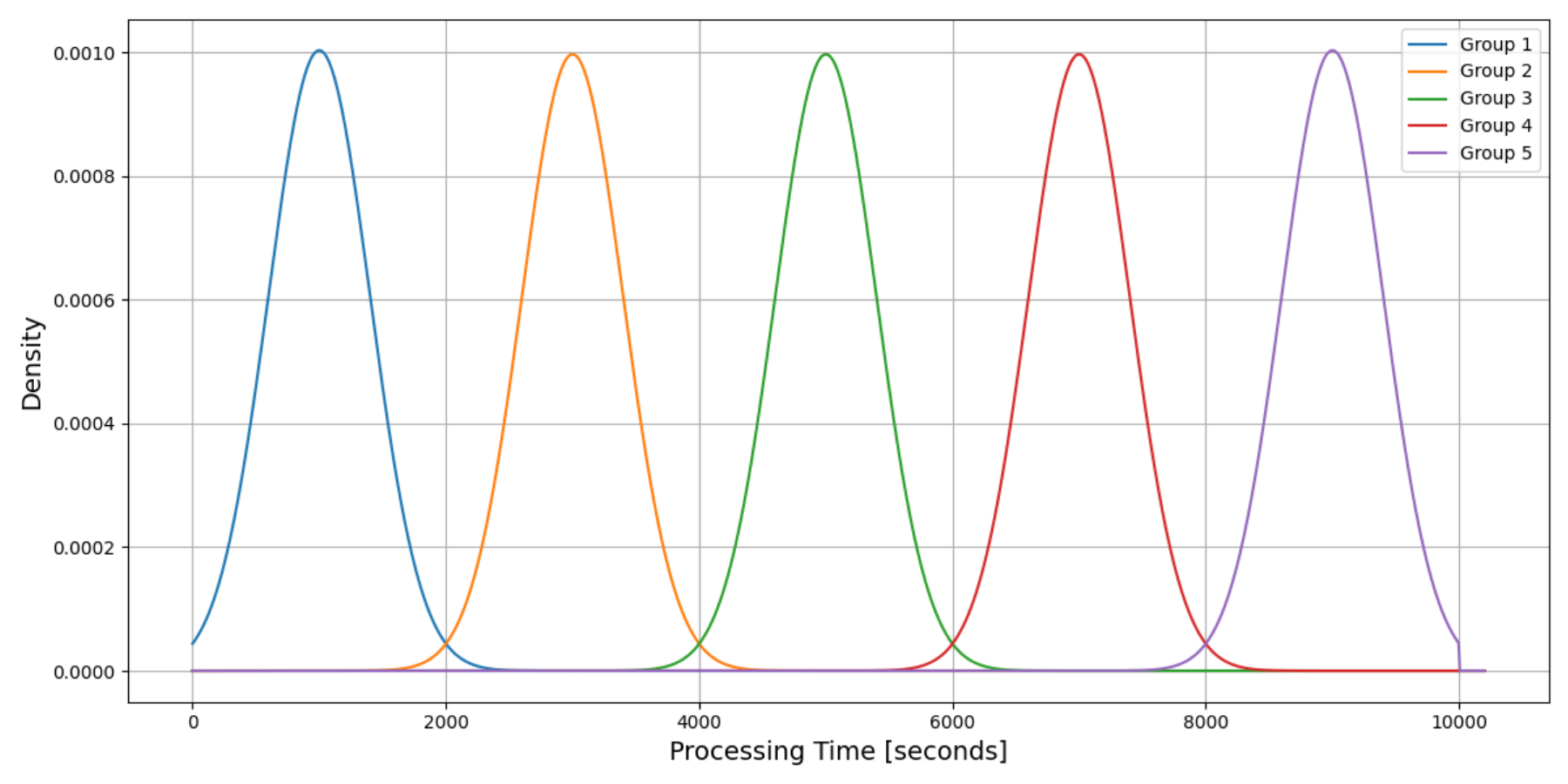}

  \caption{Estimated probability density functions (PDFs) of processing times for each job group.}
  \label{fig:normal-dist}

  \vspace{0.25em}
  {\footnotesize\emph{Note.} Each curve represents a distinct group with different mean processing durations.}
\end{figure}

Table \ref{tab:parameters-scheduling} outlines the configurations used in the experiments. The dataset contains approximately $200,000$ observations. We use $\tau = 0.9$, such that instances were sufficiently difficult.

\begin{table}[htb]
\centering
\caption{Parameter configuration for the machine learning models used in the scheduling application.\label{tab:parameters-scheduling}}
\resizebox{0.8\textwidth}{!}{%
\begin{tabular}{@{}>{\centering\arraybackslash}p{2.2cm} >{\centering\arraybackslash}p{2.2cm} >{\centering\arraybackslash}p{2.3cm} >{\raggedright\arraybackslash}p{3.3cm} >{\raggedleft\arraybackslash}p{2.0cm} >{\raggedleft\arraybackslash}p{2.3cm}@{}}
\hline
\textbf{Method} & \textbf{\begin{tabular}[c]{@{}l@{}}Instance\\ type\end{tabular}} & \textbf{\begin{tabular}[c]{@{}l@{}}Instance\\ length\end{tabular}} & \textbf{Parameter} & \textbf{Value} & \textbf{Number of parameters} \\
\hline
\multirow{5}{*}{Transformers} 
    & \multirow{5}{=}{\begin{tabular}[c]{@{}l@{}}Topologies A\\ and B\end{tabular}} 
    & \multirow{5}{*}{10} 
    & Dimension & 128 & \multirow{5}{*}{5,049,077} \\
    & & & Number of heads & 4 & \\
    & & & Encoder layers & 4 & \\
    & & & Decoder layers & 4 & \\
    & & & Learning rate & $10^{-4}$ & \\
\hline
\multirow{5}{*}{Transformers} 
    & \multirow{5}{=}{Topology C} 
    & \multirow{5}{*}{10} 
    & Dimension & 152 & \multirow{5}{*}{6,167,429} \\
    & & & Number of heads & 4 & \\
    & & & Encoder layers & 4 & \\
    & & & Decoder layers & 4 & \\
    & & & Learning rate & $10^{-4}$ & \\
\hline
\multirow{5}{*}{Transformers} 
    & \multirow{5}{=}{\begin{tabular}[c]{@{}l@{}}Topologies A\\ and B\end{tabular}} 
    & \multirow{5}{*}{20} 
    & Dimension & 152 & \multirow{5}{*}{7,604,277} \\
    & & & Number of heads & 8 & \\
    & & & Encoder layers & 6 & \\
    & & & Decoder layers & 4 & \\
    & & & Learning rate & $10^{-4}$ & \\
\hline
\multirow{5}{*}{Transformers} 
    & \multirow{5}{=}{Topology C} 
    & \multirow{5}{*}{20} 
    & Dimension & 184 & \multirow{5}{*}{9,530,165} \\
    & & & Number of heads & 8 & \\
    & & & Encoder layers & 6 & \\
    & & & Decoder layers & 4 & \\
    & & & Learning rate & $10^{-4}$ & \\
\hline
\multirow{3}{*}{LSTMs} 
    & \multirow{3}{=}{\begin{tabular}[c]{@{}l@{}}Topologies A\\ and B\end{tabular}} 
    & \multirow{3}{*}{10} 
    & Embedding dimension & 128 & \multirow{3}{*}{303,093} \\
    & & & Hidden dimension & 128 & \\
    & & & Learning rate & $5 \times 10^{-4}$ & \\
\hline
\multirow{3}{*}{LSTMs} 
    & \multirow{3}{=}{Topology C} 
    & \multirow{3}{*}{10} 
    & Embedding dimension & 152 & \multirow{3}{*}{417,909} \\
    & & & Hidden dimension & 152 & \\
    & & & Learning rate & $5 \times 10^{-4}$ & \\
\hline
\multirow{3}{*}{LSTMs} 
    & \multirow{3}{=}{\begin{tabular}[c]{@{}l@{}}Topologies A,\\ B and C\end{tabular}} 
    & \multirow{3}{*}{20} 
    & Embedding dimension & 256 & \multirow{3}{*}{1,128,437} \\
    & & & Hidden dimension & 256 & \\
    & & & Learning rate & $5 \times 10^{-4}$ & \\
\hline
\end{tabular}
}
\end{table}

We study three sets of group-based precedence constraints topologies, illustrated in Figure~\ref{fig:precedence-constraints}. Their design is guided by two key considerations. First, we avoid precedence structures that prioritize scheduling shorter jobs (i.e., those with lower group labels) before longer ones. In such cases, the shortest-processing-time-first (SPT) rule would produce the optimal solution, reducing the complexity of the learning task. To mitigate this, we assign alternating group labels that decouple processing time from precedence. Second, we design the constraint sets to exhibit increasing levels of difficulty, based on structural patterns known to challenge transformer models. In particular, we vary (i) the number of AND-type dependencies, (ii) the number of valid topological orderings, and (iii) the presence of fan-in structures, i.e., nodes with multiple predecessors. Prior work shows that transformers struggle with learning conjunctive (AND) semantics \citep{saha2020conjnli}, especially in the presence of fan-in structures and limited flexibility in valid orderings.

\begin{figure}[tb]
  \centering
  \includegraphics[width=0.85\textwidth]{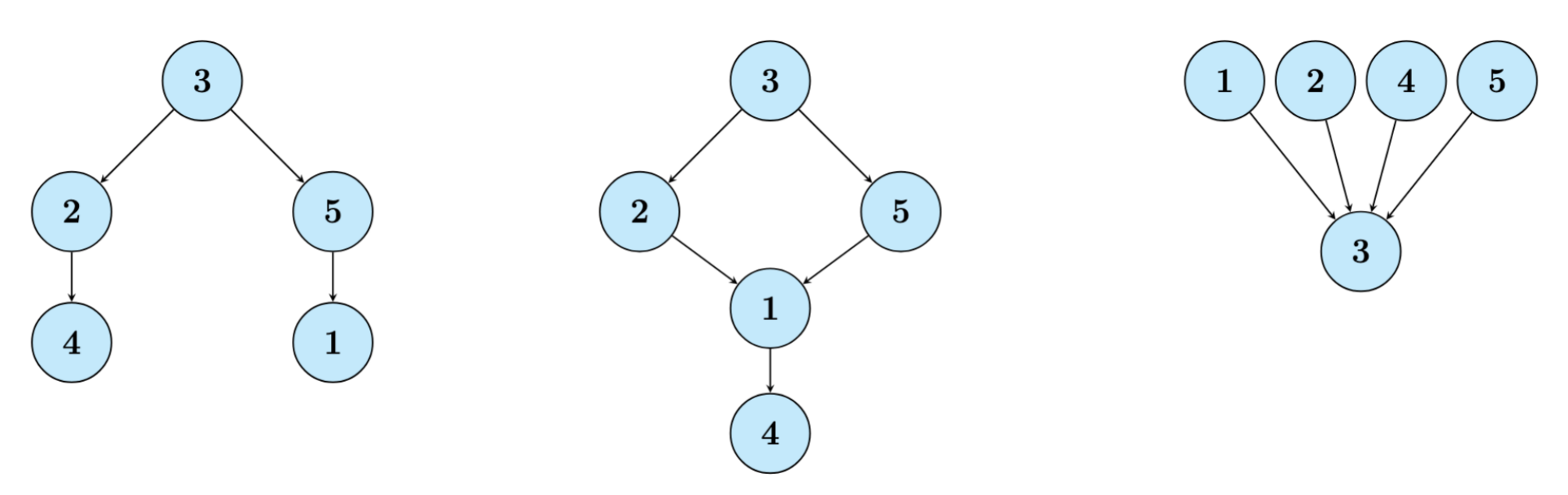}

  \caption{Group-based precedence constraints represented as directed acyclic graphs.}
  \label{fig:precedence-constraints}

  \vspace{0.25em}
  {\footnotesize\emph{Note.} Each diagram corresponds to a different set of precedence rules. From left to right: Topology A, B, and C.}
\end{figure}

The first constraint set (Topology A) consists solely of fan-out relationships and contains no AND dependencies. The second set (Topology B) introduces a single AND clause and reduces the number of valid group orderings. The third set (Topology C) presents the highest complexity: it features a fan-in structure with four AND clauses. Although more group sequences may appear feasible in this final set, correctly modeling the conjunction logic poses a significantly greater challenge to the model.

For the IO cutting-plane algorithm, we adopt the parameter values proposed by~\cite{bodur2022inverse}, with the exception of $p_0$, which we set to 500. Additionally, we use $\lambda = 10$. We also use the early stopping heuristic proposed in the paper by setting a time limit of 300 seconds for the forward problems.

\medskip

\noindent \textbf{Computational Speed.~} Table~\ref{tab:training-times-scheduling} reports training and solution generation times for the scheduling instances. For the ML–based methods, inference time scales primarily with instance length (number of jobs) and is largely insensitive to instance type. In contrast, the IO approach exhibits substantial variability, even at a fixed length and instance type, because the coefficients estimated via the cutting-plane algorithm can induce forward problems that are harder to solve.


\medskip

\begin{table}[htb]
\centering
\caption{Training (hours) and mean inference (seconds) times of competing methods on the scheduling problem instances.\label{tab:training-times-scheduling}}
\resizebox{\textwidth}{!}{%
\begin{tabular}{@{}lclrrrrrrrrrr@{}}
\hline
\multicolumn{1}{c}{\multirow{2}{*}{\textbf{\begin{tabular}[c]{@{}c@{}}Precedence\\ constraints\end{tabular}}}} &
\multirow{2}{*}{\textbf{\begin{tabular}[c]{@{}c@{}}Instance\\ length\end{tabular}}} &
\multicolumn{1}{c}{\multirow{2}{*}{\textbf{Experiment}}} &
\multicolumn{2}{c}{\textbf{Transformers}} &
\multicolumn{2}{c}{\textbf{LSTMs}} &
\multicolumn{2}{c}{\textbf{\begin{tabular}[c]{@{}c@{}}IO:\\ Linear\end{tabular}}} &
\multicolumn{2}{c}{\textbf{\begin{tabular}[c]{@{}c@{}}IO:\\ Jobs\end{tabular}}} &
\multicolumn{2}{c}{\textbf{\begin{tabular}[c]{@{}c@{}}IO:\\ Groups\end{tabular}}} \\
\cline{4-13}
\multicolumn{1}{c}{} &  & \multicolumn{1}{c}{} &
\multicolumn{1}{c}{\textbf{Training}} & \multicolumn{1}{c}{\textbf{Inference}} &
\multicolumn{1}{c}{\textbf{Training}} & \multicolumn{1}{c}{\textbf{Inference}} &
\multicolumn{1}{c}{\textbf{Training}} & \multicolumn{1}{c}{\textbf{Inference}} &
\multicolumn{1}{c}{\textbf{Training}} & \multicolumn{1}{c}{\textbf{Inference}} &
\multicolumn{1}{c}{\textbf{Training}} & \multicolumn{1}{c}{\textbf{Inference}} \\
\hline
\multirow{2}{*}{Topology A} & 10 & Full training set
& 1.51 & 0.013 & 0.64 & 0.001 & 3.50 & 0.010 & 3.50 & 0.038 & 3.50 & 0.008 \\
& 10 & Smaller training set
& 0.76 & 0.013 & 0.44 & 0.001 & 3.50 & 0.011 & 3.50 & 0.038 & 3.50 & 0.025 \\
\hline
\multirow{2}{*}{Topology B} & 10 & Full training set
& 1.56 & 0.013 & 1.00 & 0.001 & 3.50 & 0.010 & 3.50 & 0.041 & 3.50 & 0.008 \\
& 10 & Smaller training set
& 0.79 & 0.013 & 0.77 & 0.001 & 3.50 & 0.011 & 3.50 & 0.036 & 3.50 & 0.021 \\
\hline
\multirow{2}{*}{Topology C} & 10 & Full training set
& 2.61 & 0.013 & 0.36 & 0.001 & 3.50 & 0.013 & 3.50 & 0.039 & 3.50 & 0.036 \\
& 10 & Smaller training set
& 0.83 & 0.013 & 0.56 & 0.001 & 3.50 & 0.009 & 3.50 & 0.038 & 3.50 & 0.035 \\
\hline
\multirow{2}{*}{Topology A} & 20 & Full training set
& 9.69 & 0.025 & 1.26 & 0.002 & 12.00 & 0.600 & 12.00 & 144.000 & 12.00 & 0.048 \\
& 20 & Smaller training set
& 1.42 & 0.025 & 1.15 & 0.002 & 12.00 & 0.120 & 12.00 & 1.042 & 12.00 & 0.048 \\
\hline
\multirow{2}{*}{Topology B} & 20 & Full training set
& 4.31 & 0.025 & 2.07 & 0.002 & 12.00 & 0.587 & 12.00 & 257.142 & 12.00 & 0.064 \\
& 20 & Smaller training set
& 1.61 & 0.025 & 1.98 & 0.002 & 12.00 & 0.122 & 12.00 & 1.113 & 12.00 & 0.043 \\
\hline
\multirow{2}{*}{Topology C} & 20 & Full training set
& 2.97 & 0.025 & 1.08 & 0.002 & 12.00 & 2.402 & 12.00 & 3.209 & 12.00 & 46.153 \\
& 20 & Smaller training set
& 2.14 & 0.025 & 1.15 & 0.002 & 12.00 & 7.207 & 12.00 & 1.338 & 12.00 & 2.199 \\
\hline
\end{tabular}
}
\end{table}

\noindent \textbf{Results.~} Table~\ref{tab:precedence-results} reports the percentages of precedence constraint satisfaction for all methods across the different experiments and instances. ML-based methods substantially outperform the IO methods, even though they are trained without any information about the precedence constraints. Transformer models generally achieve the highest satisfaction rates among all approaches.


In contrast, for the IO method, the presence of precedence constraints is hinted at through the use of the second and third objective functions. Notably, the method performs comparatively better on Topology C instances than on the others, likely because it only needs to learn to assign smaller coefficients to jobs in group 3, whereas in the other topologies this relationship is less straightforward.

\begin{table}
\centering
\small
\caption{Percentage (\%) of solutions that completely satisfy precedence constraints for the single-machine scheduling problem.\label{tab:precedence-results}}
\resizebox{\textwidth}{!}{%
\begin{tabular}{@{}lclrrrrrrr@{}}
\hline
\textbf{\begin{tabular}[c]{@{}l@{}}Precedence\\constraints\end{tabular}} & 
\textbf{\begin{tabular}[c]{@{}c@{}}Instance\\length\end{tabular}} & 
\textbf{Experiment} & 
\textbf{Transformers} & 
\textbf{LSTMs} & 
\textbf{\begin{tabular}[c]{@{}c@{}}IO:\\Linear\end{tabular}} & 
\textbf{\begin{tabular}[c]{@{}c@{}}IO:\\Jobs\end{tabular}} & 
\textbf{\begin{tabular}[c]{@{}c@{}}IO:\\Groups\end{tabular}} & 
\textbf{Random} \\
\hline
\multirow{2}{*}{Topology A}  
    & 10 & Full training set       & 100.00 & 99.60 & 1.20 & 3.10  & 2.40  & \multirow{2}{*}{0.90} \\
    & 10 & Smaller training set    & 99.80  & 93.40 & 1.90 & 1.90 & 45.10 &                         \\
\hline
\multirow{2}{*}{Topology B} 
    & 10 & Full training set       & 100.00 & 99.90 & 1.40 & 2.20  & 1.70  & \multirow{2}{*}{1.20} \\
    & 10 & Smaller training set    & 100.00 & 95.50 & 1.00 & 2.40  & 16.90 &                         \\
\hline
\multirow{2}{*}{Topology C}  
    & 10 & Full training set       & 96.50  & 100.00 & 30.30 & 60.90 & 99.60 & \multirow{2}{*}{13.20} \\
    & 10 & Smaller training set    & 100.00 & 97.00 & 11.10 & 61.90 & 65.10 &                         \\
\hline
\multirow{2}{*}{Topology A}  
    & 20 & Full training set       & 100.00 & 99.20 & 0.00 & \begin{tabular}[c]{@{}r@{}}0.00\\(5.00\%)\end{tabular}  & 6.90  & \multirow{2}{*}{0.00} \\
    & 20 & Smaller training set    & 99.30  & 90.60 & 0.00 & 0.00 & 1.80  &                         \\
\hline
\multirow{2}{*}{Topology B} 
    & 20 & Full training set       & 99.70  & 100.00 & 0.00 & \begin{tabular}[c]{@{}r@{}}0.00\\(2.80\%)\end{tabular} & 21.80 & \multirow{2}{*}{0.00} \\
    & 20 & Smaller training set    & 100.00 & 97.10  & 0.00 & 0.00 & 0.40  &                         \\
\hline
\multirow{2}{*}{Topology C}  
    & 20 & Full training set       & 99.80  & 100.00 & 49.40 & 38.70 & \begin{tabular}[c]{@{}r@{}}0.64\\(15.60\%)\end{tabular} & \multirow{2}{*}{0.50} \\
    & 20 & Smaller training set    & 100.00 & 97.30 & 43.30 & 34.40 & 49.00 &                         \\
\hline
\end{tabular}
}
\caption*{\footnotesize The percentage of solved instances shown in parentheses when it is less than 100\%. For those instances, the percentage reported is among the solved solutions.}
\end{table}

Figures~\ref{fig:edit-boxplot-10} and ~\ref{fig:edit-boxplot-20} present the optimality gaps for instances of length 10 and 20, respectively. Transformers achieve by far the smallest edit distances, followed by LSTMs. Consistent with the results on precedence constraint satisfaction, the IO method performs substantially better on Topology C instances. Interestingly, Topology C is also the topology where ML-based methods struggle the most. As previously noted, this topology poses a particular challenge for transformers due to its multiple AND clauses and fan-in structure.

\begin{figure}[htb]
  \centering

  \includegraphics[width=\textwidth]{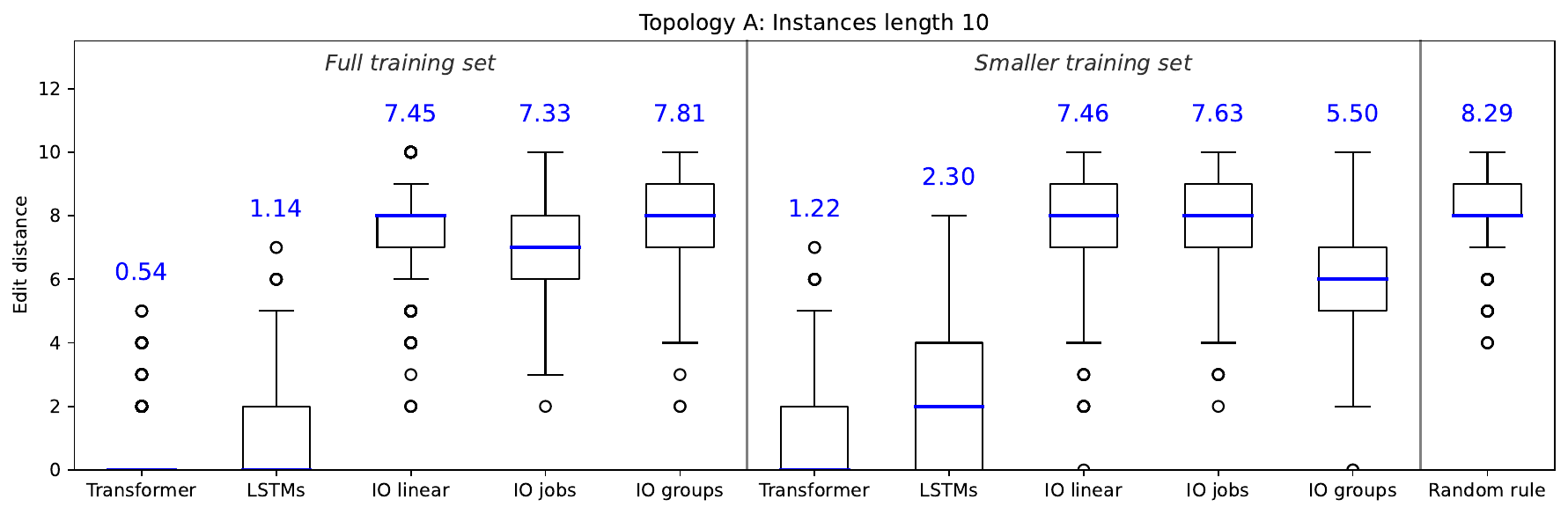}

  \vspace{0.8em}

  \includegraphics[width=\textwidth]{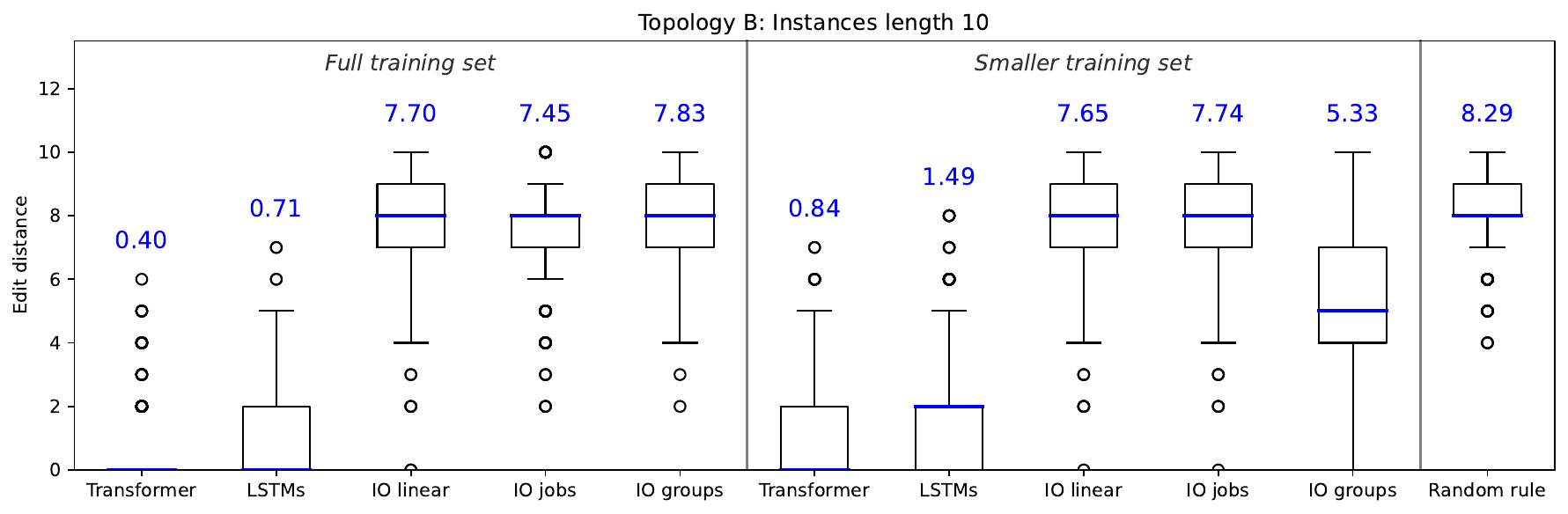}

  \vspace{0.8em}

  \includegraphics[width=\textwidth]{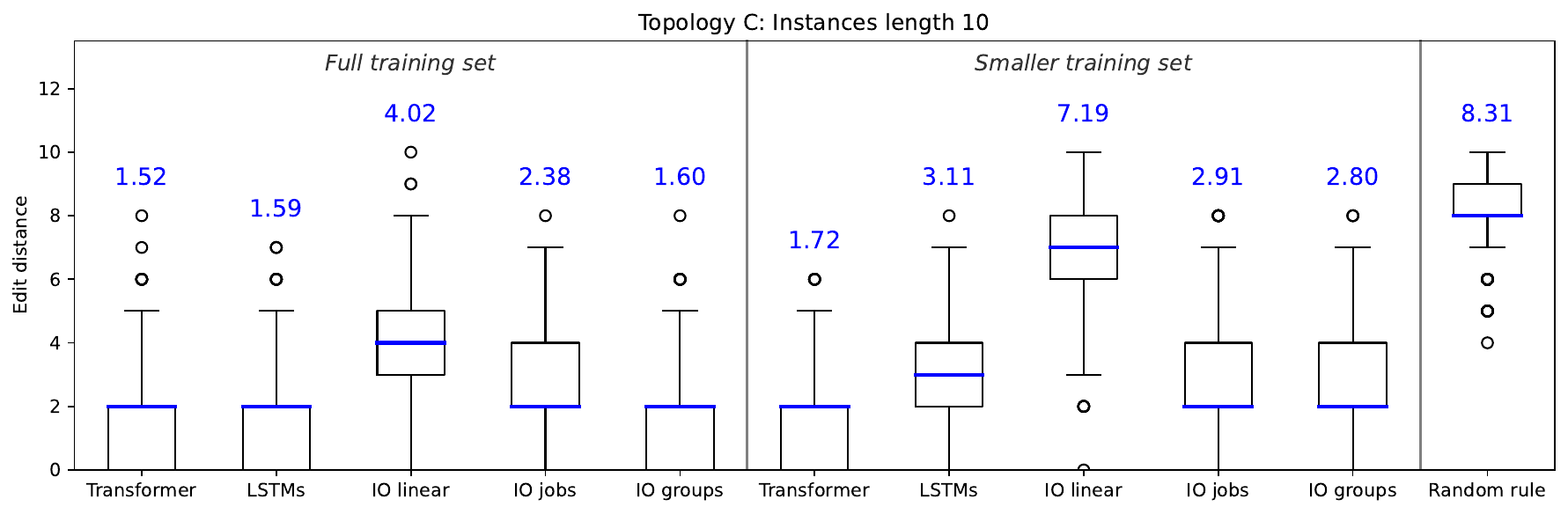}

  \caption{Edit distances length 10 for the single-machine scheduling problem.}
  \label{fig:edit-boxplot-10}

  \vspace{0.25em}
  {\footnotesize\emph{Note.} Means are displayed on top of the boxes.}
\end{figure}

\begin{figure}[htb]
  \centering

  \includegraphics[width=\textwidth]{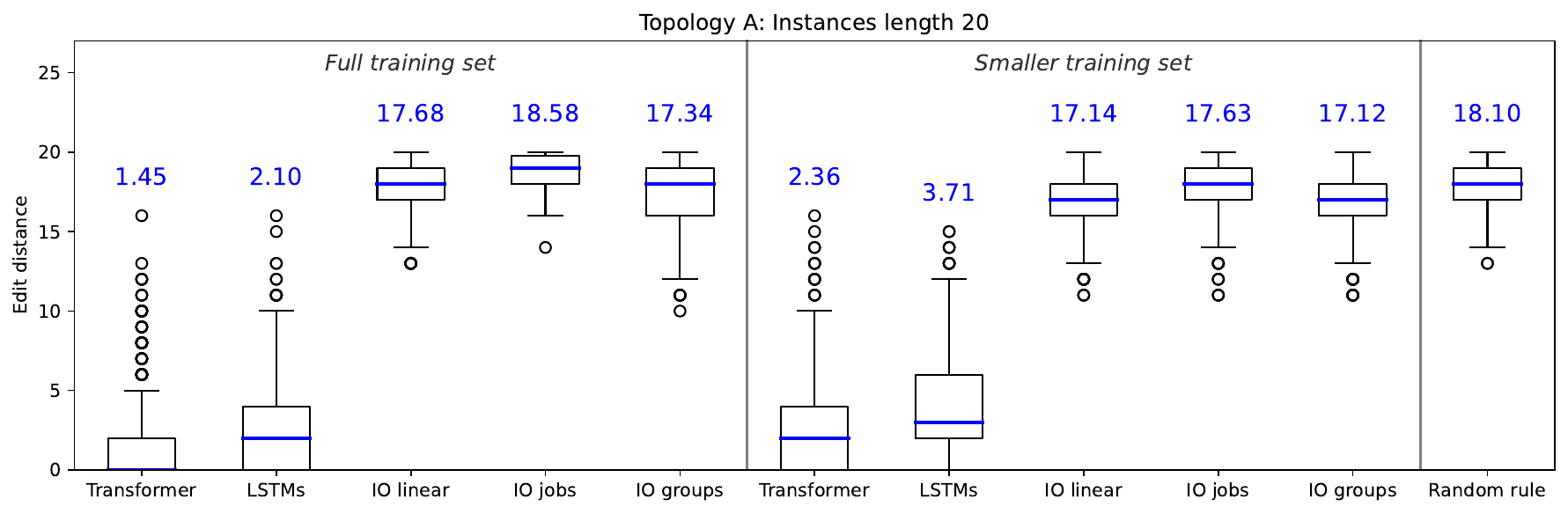}

  \vspace{0.7em}

  \includegraphics[width=\textwidth]{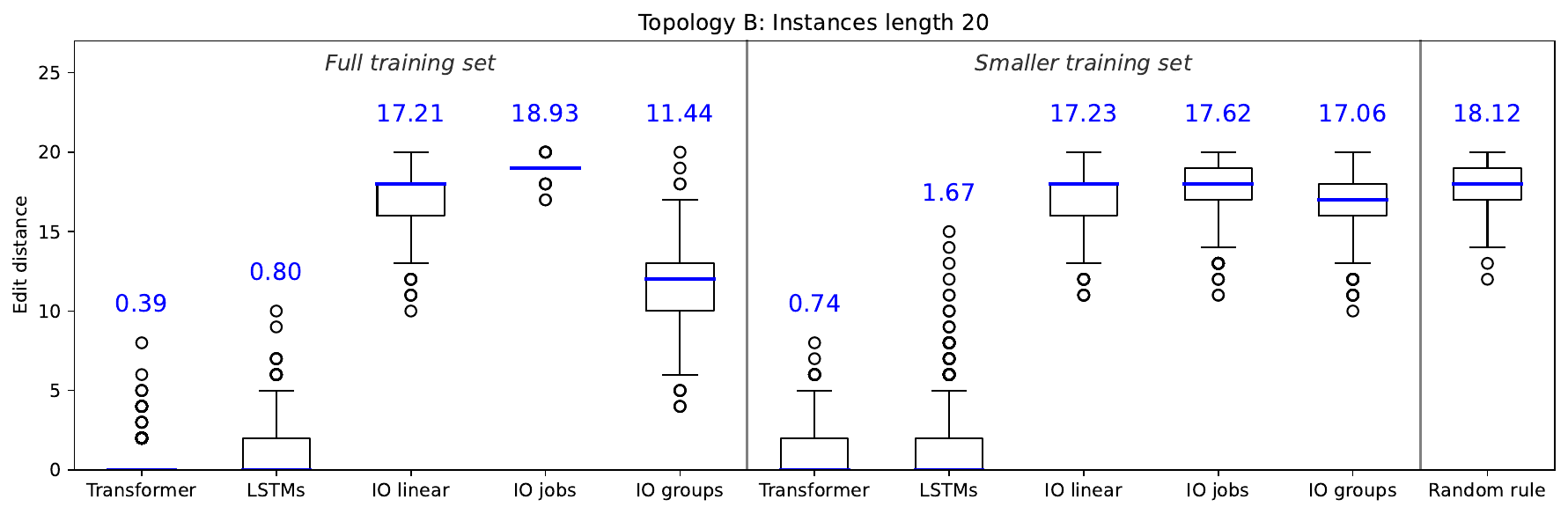}

  \vspace{0.7em}

  \includegraphics[width=\textwidth]{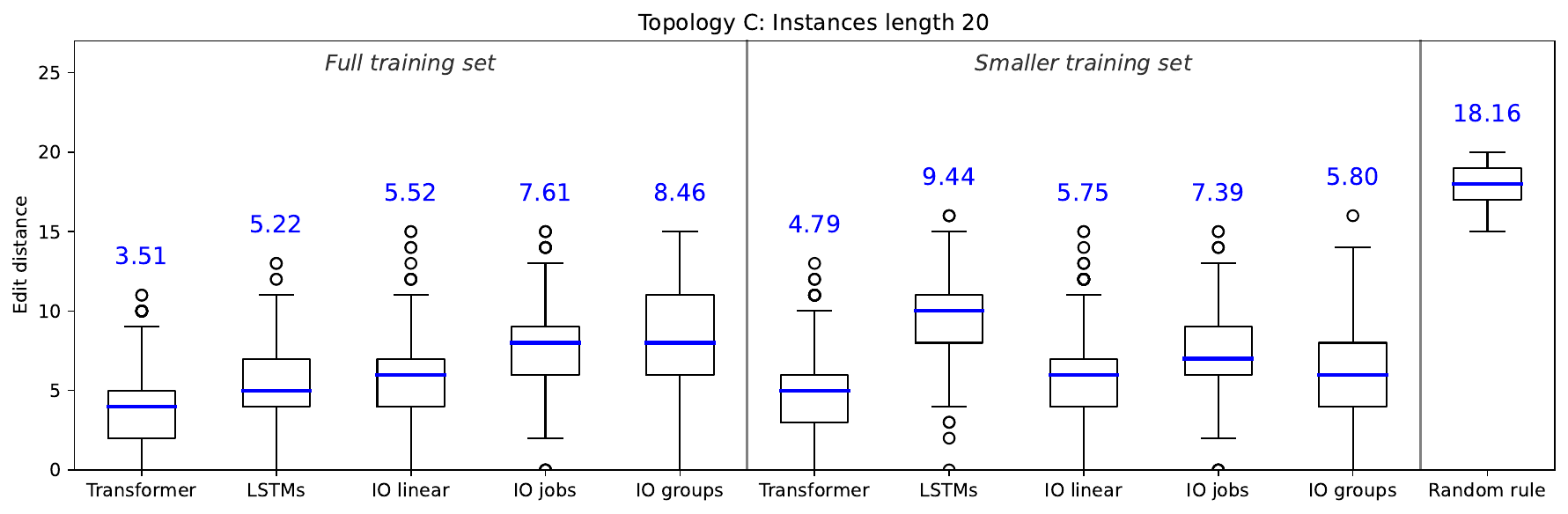}

  \caption{Edit distances length 20 for the single-machine scheduling problem.}
  \label{fig:edit-boxplot-20}

  \vspace{0.25em}
  {\footnotesize\emph{Note.} Means are displayed on top of the boxes.}
\end{figure}

Although the edit distance metric serves as a useful proxy for method performance, our primary goal is not to exactly replicate the optimal solutions. Since the task is an optimization problem, the focus is on satisfying precedence constraints and achieving small optimality gaps by learning from the optimal solutions. Figure~\ref{fig:gaps-boxplot-10} and and Figure~\ref{fig:gaps-boxplot-20} present the optimality gaps for instances of length 10 and 20, respectively. The performance trends mirror those observed with the other metrics.

\begin{figure}[htb]
  \centering

  \includegraphics[width=\textwidth]{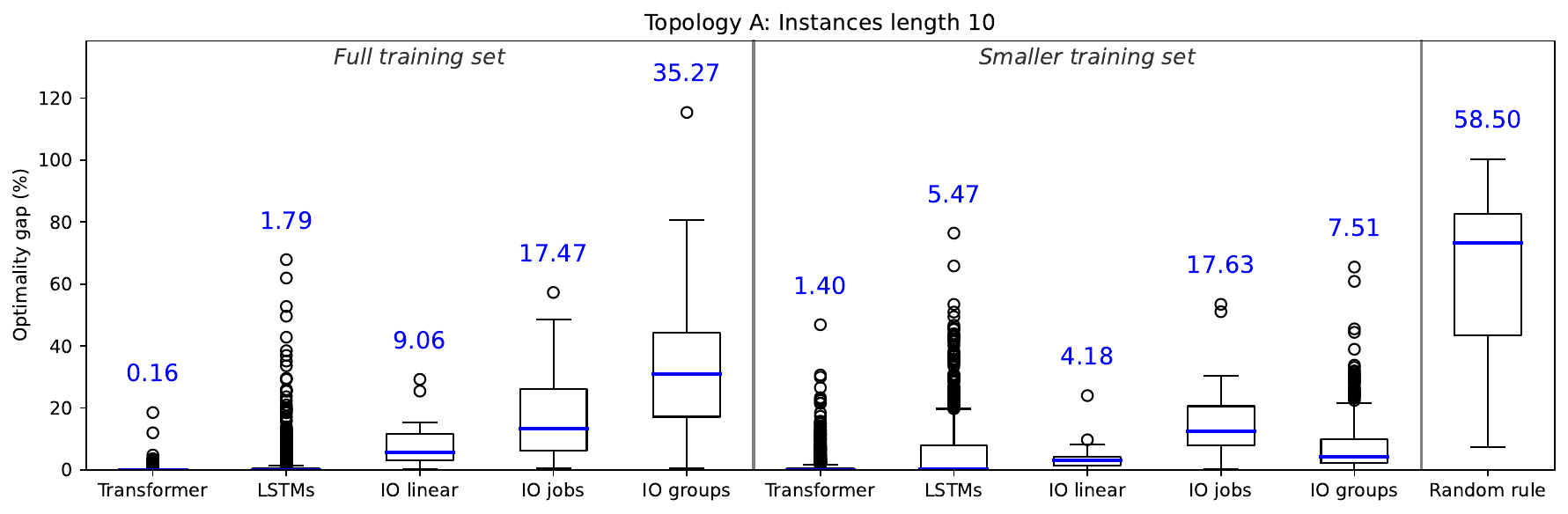}

  \vspace{0.7em}

  \includegraphics[width=\textwidth]{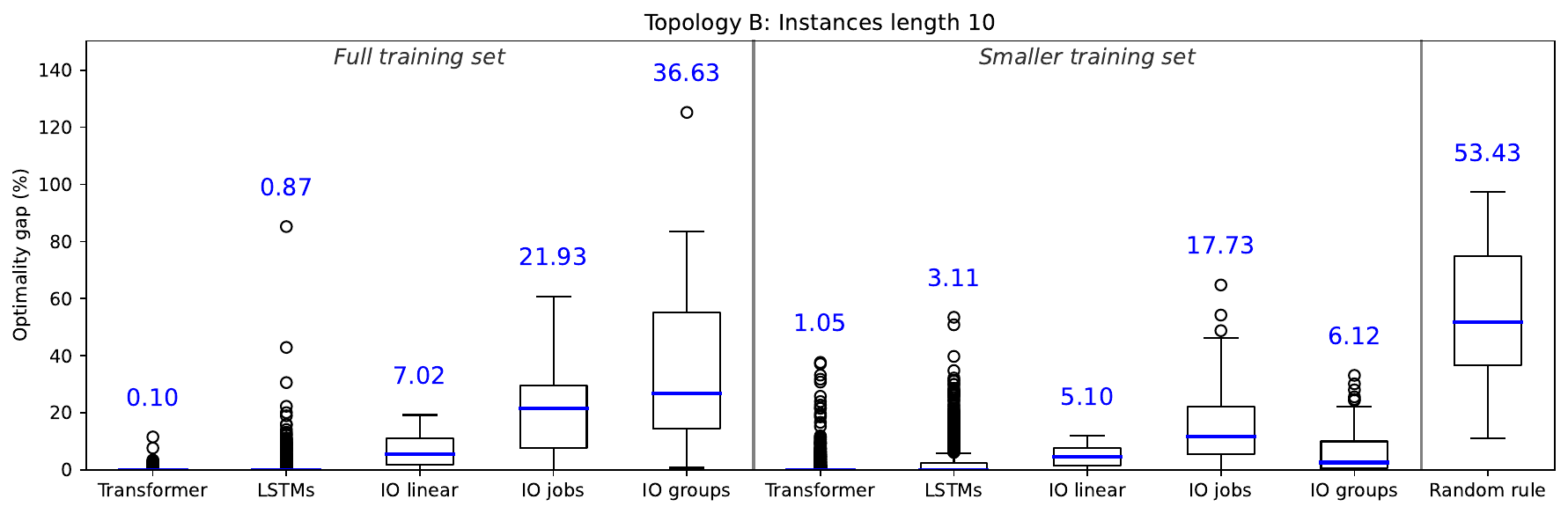}

  \vspace{0.7em}

  \includegraphics[width=\textwidth]{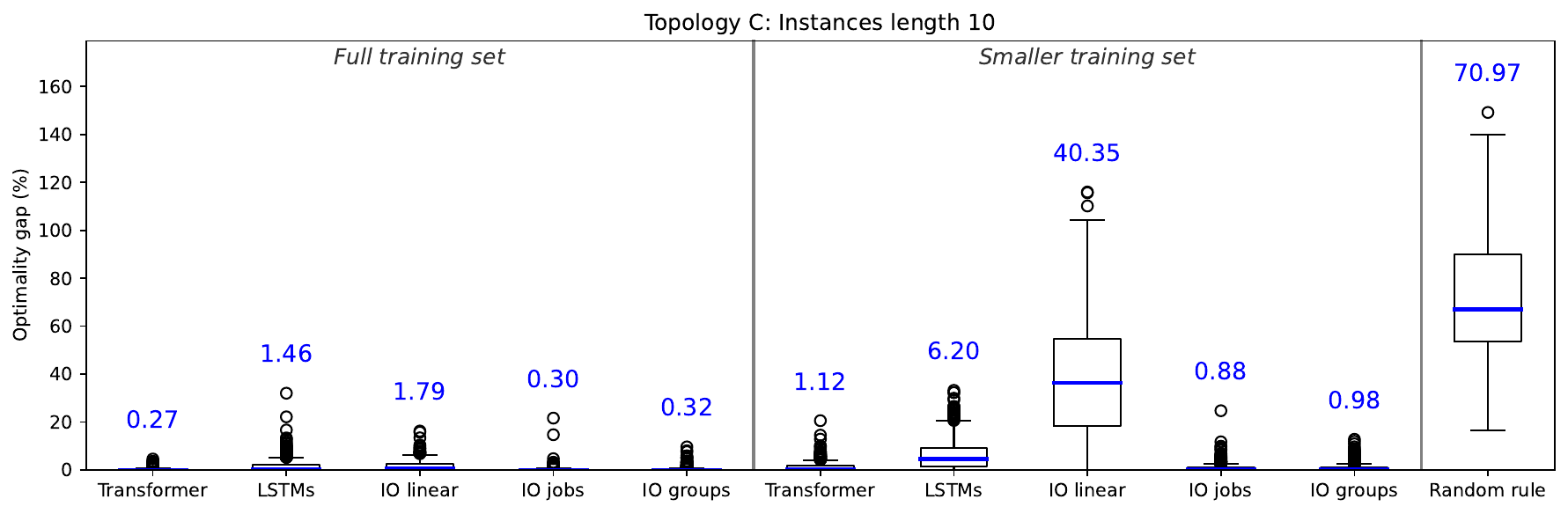}

  \caption{Optimality gaps (\%) for instances of length 10 for the single-machine scheduling problem.}
  \label{fig:gaps-boxplot-10}

  \vspace{0.25em}
  {\footnotesize\emph{Note.} Means are displayed on top of the boxes. These results only consider solutions that satisfy precedence constraints.}
\end{figure}

\begin{figure}[htb]
  \centering

  \includegraphics[width=\textwidth]{figs/gaps_type1_20_scheduling_style.pdf}

  \vspace{0.7em}

  \includegraphics[width=\textwidth]{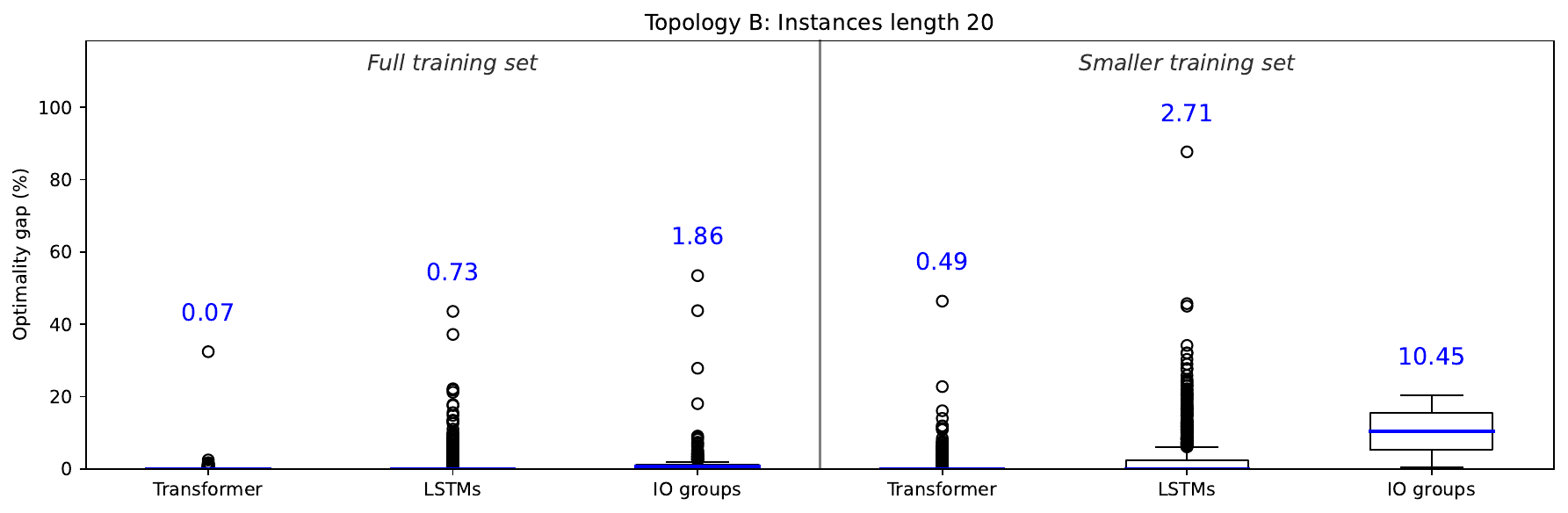}

  \vspace{0.7em}

  \includegraphics[width=\textwidth]{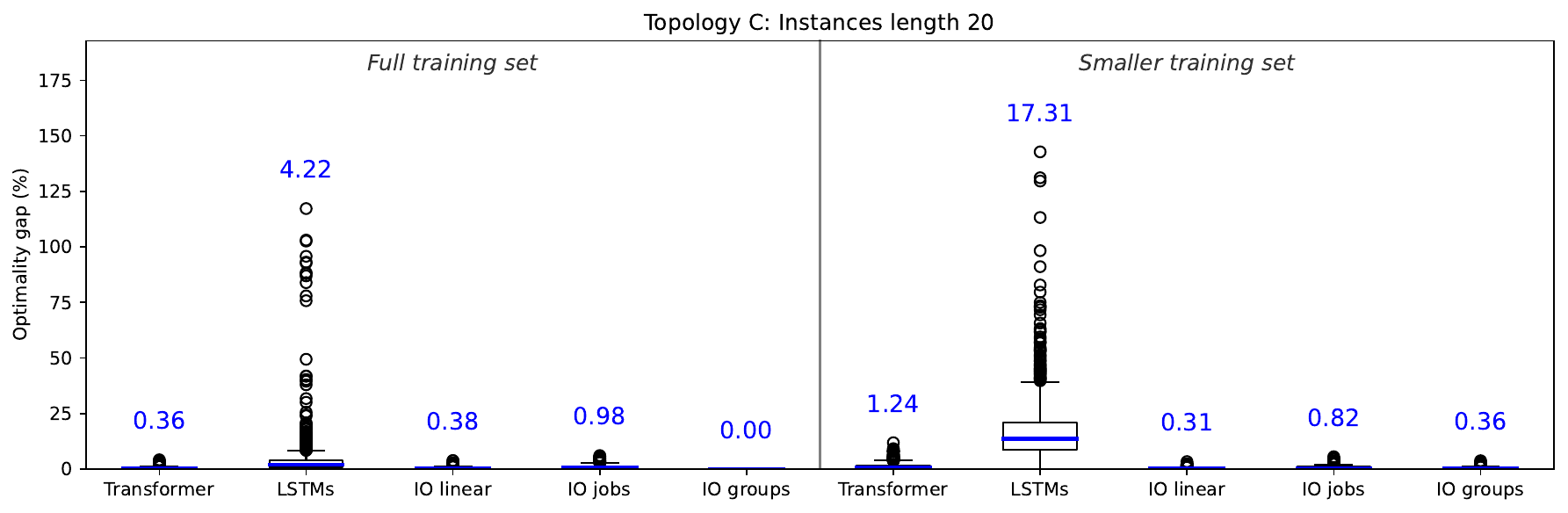}

  \caption{Optimality gaps (\%) for instances of length 20 for the single-machine scheduling problem.}
  \label{fig:gaps-boxplot-20}

  \vspace{0.25em}
  {\footnotesize\emph{Note.} Means are displayed on top of the boxes. These results only consider solutions that satisfy precedence constraints. We do not report the boxplot for the linear IO approach for type A and type B instances because it did not generate any feasible solutions.}
\end{figure}

Additionally, Table \ref{tab:gap0-scheduling} reports the percentage of optimally solved instances for the different topologies and configurations.

\clearpage
\noindent
\begin{minipage}[t]{\textwidth}
\centering
\small
\captionof{table}{Percentage (\%) of generated solutions solved to optimality for the single-machine scheduling problem with respect to the testing set.\label{tab:gap0-scheduling}}

\begin{tabular}{@{}lclrrrrrrr@{}}
\hline
\textbf{\begin{tabular}[c]{@{}l@{}}Precedence\\constraints\end{tabular}} & 
\textbf{\begin{tabular}[c]{@{}c@{}}Instance\\length\end{tabular}} & 
\textbf{Experiment} & 
\textbf{Transformers} & 
\textbf{LSTMs} & 
\textbf{\begin{tabular}[c]{@{}c@{}}IO:\\First\end{tabular}} & 
\textbf{\begin{tabular}[c]{@{}c@{}}IO:\\Second\end{tabular}} & 
\textbf{\begin{tabular}[c]{@{}c@{}}IO:\\Third\end{tabular}} & 
\textbf{Random} \\
\hline
\multirow{2}{*}{Topology A}  
    & 10 & Full training set       & 77.00 & 57.00 & 0.00 & 0.00 & 0.00  & \multirow{2}{*}{0.00} \\
    & 10 & Smaller training set    & 54.60 & 29.10 & 0.10 & 0.00  & 0.20  &                         \\
\hline
\multirow{2}{*}{Topology B} 
    & 10 & Full training set       & 82.80 & 71.50 & 0.20 & 0.00  & 0.00  & \multirow{2}{*}{0.00} \\
    & 10 & Smaller training set    & 67.60 & 48.30 & 0.10 & 0.10  & 0.80  &                         \\
\hline
\multirow{2}{*}{Topology C}  
    & 10 & Full training set       & 38.50 & 39.50  & 5.80 & 23.30  & 33.90 & \multirow{2}{*}{0.00} \\
    & 10 & Smaller training set    & 35.70 & 9.80  & 0.10 & 14.30  & 16.70 &                         \\
\hline
\multirow{2}{*}{Topology A}  
    & 20 & Full training set       & 59.20 & 49.20  & * & \begin{tabular}[c]{@{}r@{}}*\\(5.00\%)\end{tabular}  & 0.00  & \multirow{2}{*}{*} \\
    & 20 & Smaller training set    & 42.30 & 21.60  & *  & *  & 0.00  &                         \\
\hline
\multirow{2}{*}{Topology B} 
    & 20 & Full training set       & 83.40 & 73.10  & *  & \begin{tabular}[c]{@{}r@{}}*\\(2.80\%)\end{tabular}  & 0.00  & \multirow{2}{*}{*} \\
    & 20 & Smaller training set    & 72.60 & 51.90  & *  & * & 0.00  &                         \\
\hline
\multirow{2}{*}{Topology C}  
    & 20 & Full training set       & 13.70 & 3.80  & 2.80 & 0.50 & \begin{tabular}[c]{@{}r@{}}0.10\\(15.60\%)\end{tabular} & \multirow{2}{*}{0.00} \\
    & 20 & Smaller training set    & 5.60  & 0.10  & 2.40 & 0.50 & 2.70  &                         \\
\hline
\end{tabular}

\par\medskip
{\footnotesize The percentage of solved instances is shown in parentheses when it is less than 100\%. * corresponds to experiments where the method did not generate any solution that satisfies precedence constraints and, therefore, we did not compute their optimality gap.}
\end{minipage}

\end{document}